\documentclass[a4paper]{amsart}
\usepackage{graphicx}
\usepackage{amssymb}
\usepackage{mathtools}
\usepackage{stmaryrd}
\usepackage{dsfont}
\usepackage[mathcal]{euscript}
\usepackage{tikz}
\usepackage{tikz-cd}
\usepackage{hyperref}
\hypersetup{%
  bookmarksnumbered=true,%
  colorlinks=true,%
  linkcolor=blue,%
  citecolor=blue,%
  filecolor=blue,%
  menucolor=blue,%
  urlcolor=blue,%
  bookmarksopen=true,%
  bookmarksdepth=2,%
  pageanchor=true}

\title{The Boolean spectrum of a Grothendieck category}

\author{Henning Krause}
\address{Fakult\"at f\"ur Mathematik\\
Universit\"at Bielefeld\\ D-33501 Bielefeld\\ Germany}
\email{hkrause@math.uni-bielefeld.de}

\theoremstyle{plain}
\newtheorem{thm}{Theorem}[section]
\newtheorem*{Thm}{Theorem}
\newtheorem{prop}[thm]{Proposition}
\newtheorem{lem}[thm]{Lemma} 

\newtheorem{cor}[thm]{Corollary}

\theoremstyle{definition}
\newtheorem{defn}[thm]{Definition}
\newtheorem{exm}[thm]{Example}
\newtheorem{probl}[thm]{Problem}

\theoremstyle{remark}
\newtheorem{rem}[thm]{Remark}

\numberwithin{equation}{section}

\newcommand{\Cl}{\operatorname{Cl}}
\newcommand{\Clop}{\operatorname{Clop}}

\newcommand{\Coker}{\operatorname{Coker}}

\newcommand{\End}{\operatorname{End}}

\newcommand{\Ext}{\operatorname{Ext}}

\newcommand{\fp}{\operatorname{fp}}
\newcommand{\Fpinj}{\operatorname{Fpinj}}

\newcommand{\Hom}{\operatorname{Hom}}

\newcommand{\id}{\operatorname{id}}

\renewcommand{\Im}{\operatorname{Im}}

\newcommand{\Ind}{\operatorname{Ind}}

\newcommand{\Inj}{\operatorname{Inj}}

\newcommand{\Ker}{\operatorname{Ker}}

\newcommand{\mmod}{\operatorname{mod}}
\newcommand{\Mod}{\operatorname{Mod}}

\newcommand{\Noeth}{\operatorname{Noeth}}
\newcommand{\Ob}{\operatorname{Ob}}

\newcommand{\PInj}{\operatorname{PInj}}

\newcommand{\PSpc}{\operatorname{PSpc}}
\newcommand{\PSpec}{\operatorname{PSpec}}

\newcommand{\rad}{\operatorname{rad}}
\newcommand{\Rad}{\operatorname{Rad}}

\newcommand{\soc}{\operatorname{soc}}

\newcommand{\Sp}{\operatorname{Sp}}
\newcommand{\Spc}{\operatorname{Spc}}
\newcommand{\Spec}{\operatorname{Spec}}

\newcommand{\supp}{\operatorname{supp}}
\newcommand{\suppex}{\operatorname{supp}_{\mathrm{ex}}}

\newcommand{\Ab}{\mathrm{Ab}}

\newcommand{\Ess}{\mathrm{Ess}} 
\newcommand{\ex}{\mathrm{ex}}

\newcommand{\op}{\mathrm{op}}

\newcommand{\BRing}{\mathbf{BRing}} 
\newcommand{\Top}{\mathbf{Top}} 
\newcommand{\two}{\mathbf{2}}

\newcommand{\iso}{\xrightarrow{\raisebox{-.4ex}[0ex][0ex]{$\scriptstyle{\sim}$}}}

\newcommand{\longiso}{\xrightarrow{\ \raisebox{-.4ex}[0ex][0ex]{$\scriptstyle{\sim}$}\ }}

\newcommand{\lto}{\longrightarrow}
\newcommand{\smatrix}[1]{\left[\begin{smallmatrix}#1\end{smallmatrix}\right]}

\newcommand{\xto}{\xrightarrow}
\newcommand*{\intref}[2]{\def\tmp{#1}\ifx\tmp\empty\hyperref[#2]{\ref*{#2}}\else\hyperref[#2]{#1~\ref*{#2}}\fi}

\def\A{\mathcal A} 
\def\B{\mathcal B} 
\def\C{\mathcal C}
\def\D{\mathcal D} 
 
\def\F{\mathcal F}

\def\calS{\mathcal S} 
\def\T{\mathcal T} 
\def\U{\mathcal U}
\def\V{\mathcal V}
\def\X{\mathcal X}

\def\bfB{\mathbf B} 
 
\def\bfD{\mathbf D} 
\def\bfK{\mathbf K}
\def\bfL{\mathbf L} 
\def\bfP{\mathbf P} 
\def\bfR{\mathbf R}
\def\bfS{\mathbf S}

\def\bbQ{\mathbb Q}

\def\bbZ{\mathbb Z}

\newcommand{\frp}{\mathfrak{p}}

\def\a{\alpha}
\def\b{\beta}

\def\p{\phi}
\def\s{\sigma}
\def\t{\tau}

\def\La{\Lambda}

\begin{document}

\keywords{Boolean lattice, definable subcategory, Grothendieck category, spectral
  category, support, Ziegler spectrum}

\subjclass[2020]{18E10 (primary), 16D70, 16E50, 18E40, 18E45 (secondary)}

\begin{abstract}
  A notion of support for objects in any Grothendieck category is
  introduced. This is based on the spectral category of a Grothendieck
  category and uses its Boolean lattice of localising
  subcategories. The support provides a classification of all
  subcategories that are closed under arbitrary coproducts,
  subobjects, and essential extensions.  There is also a notion of
  exact support which classifies certain thick subcategories.  As an
  application, the coproduct decompositions of objects are described
  in terms of Boolean lattices. Also, for any ring Crawley-Boevey's
  correspondence between definable subcategories of modules and closed
  subsets of the Ziegler spectrum is extended.
\end{abstract}

\maketitle

\section{Introduction}

The spectrum $\Sp\A$ of a Grothendieck category $\A$ is the set of
isomorphism classes of indecomposable injective objects. This
definition is due to Gabriel, and he uses this spectrum when he
reconstructs a noetherian scheme from its category of quasi-coherent
sheaves \cite{Ga1962}. A finer invariant is introduced in work of
Gabriel and Oberst \cite{GO1966}. The spectral category $\Spec\A$ is
again a Grothendieck category with the additional property that every
short exact sequence is split exact. There is a canonical functor
$\A\to\Spec\A$ which identifies $\Sp\A$ with the isomorphism classes of
simple objects in $\Spec\A$.

In this work we study the lattice $\bfL(\Spec\A)$ of localising
subcategories of $\Spec\A$, which one may think of as an intermediate
between $\Sp\A$ and $\Spec\A$.  This is actually a complete
Boolean lattice; so it identifies via Stone duality with the Boolean
lattice of clopen sets of its spectrum
\[\Spc\A:=\Spec(\bfL(\Spec\A))\] which we call the \emph{Boolean
  spectrum} of $\A$. There is a canonical embedding
\[\Sp\A\lhook\joinrel\xrightarrow{\ \ \ }\Spc\A\]
which is a bijection if and only if $\Spec\A$ is \emph{discrete},
which means that every object decomposes into a coproduct of simple
objects. This happens, for example, when $\A$ is locally
noetherian, or (more generally) when the Krull dimension of $\A$ is
defined \cite[IV.1]{Ga1962}.

To each object $X\in\A$ we assign its support
$\supp(X)\subseteq\Spc\A$, and this sets up a bijection
\begin{equation}\label{eq:supp(C)}
  \C \longmapsto\supp(\C)=\bigvee_{X\in\C}\supp(X)
\end{equation}
between the essentially closed subcategories of $\A$ and the clopen subsets of
$\Spc\A$ (Theorem~\ref{th:spectral}), where a full subcategory is
called \emph{essentially closed} if it is closed under arbitrary coproducts,
subobjects, and essential extensions. Let $\langle X\rangle$ denote
the smallest essentially closed subcategory containing $X$. The support enjoys
the following universal property, which is the analogue of a property
of support for tensor triangulated categories due to Balmer
\cite{Ba2005}. We do not assume any monoidal structure, but our
analogue seems reasonable given that one has the identity
$\langle X\otimes Y\rangle=\langle X\rangle\cap \langle Y\rangle$ in
the tensor triangulated setting.

\begin{Thm}[Theorem~\ref{th:support}]
  The map $\supp\colon\Ob\A\to\Clop(\Spc\A)$ has the following properties:
\begin{enumerate}
\item[(S$0$)] $\supp(X)=\varnothing$ for $X=0$.
\item[(S$1$)] $\supp(X)=\Spc\A$ for every generator $X$ of $\A$.
\item[(S$\vee$)] $\supp(X)\cup\supp(Y)=\supp(X\oplus Y)$ for all objects $X,Y$.
\item[(S$\wedge$)] $\supp(X)\cap\supp(Y)=\supp(Z)$ for all objects $X,Y,Z$ 
  with $\langle X\rangle\cap \langle Y\rangle=\langle Z\rangle$.
\end{enumerate}
Moreover, for any pair $(T,\s)$ which consists of a topological space
$T$ and a map $\s\colon\Ob\A\to\Clop(T)$ satisfying the analogue of
the above conditions, there exists a unique continuous map
$f\colon T\to \Spc\A$ such that $\s(X)=f^{-1}(\supp(X))$ for all $X$.
\end{Thm}

The local structure of an object can be described in terms of its
support. More precisely, we introduce for each injective object $X$ a
complete Boolean lattice $\bfD(X)$ which is given by the clopen subsets of
$\supp(X)$; its elements parameterise the direct summands of $X$ up to
an appropriate equivalence relation and there is a canonical
isomorphism
\[\bfD(X)\longiso\bfB(\End(X)/J(\End(X))),\] where $\bfB(\La)$ denotes 
the lattice of central idempotents of a ring $\La$
(Theorem~\ref{th:decomp}). The ring $\End(X)/J(\End(X))$ is von
Neumann regular; so an elaborated theory of decompositions and types
applies \cite{Ka1951,MvN1936}. It is interesting to note that any
complete Boolean lattice can be realised as $\bfD(X)$ for some
injective object $X$ and some appropriate Grothendieck category.

The support of an object has the useful property that $\supp(X)=\varnothing$ implies
$X=0$. This  one cannot expect when using $\Sp\A$ as target for
a support function, because $\Sp\A$ may be empty. Another option for
defining support would be to use as target the frame $\bfL(\A)$ of
localising subcategories of $\A$. However, its set of points may be
empty as well \cite{St2024}. The following result shows that $\Spc\A$
is an invariant which encompasses both $\Sp\A$ and $\bfL(\A)$.

\begin{Thm}[Theorem~\ref{th:spectral-localising}]
  Taking a localising subcategory to its support induces an embedding
  of frames
\[\bfL(\A)\hookrightarrow\bfL(\Spec\A)\iso\Clop(\Spc\A).\]
\end{Thm}

For example, when $\A=\Mod\La$ equals the module category of a commutative
noetherian ring $\La$, then $\Sp\A\iso\Spc\A$ identifies with the set
of prime ideals $\Spec\La$, and the above correspondence
\eqref{eq:supp(C)} extends Gabriel's bijection between the localising
subcategories of $\Mod\La$ and the specialisation closed subsets of
$\Spec\La$, which is one of the ingredients for the reconstruction of a
noetherian scheme \cite{Ga1962}.

The work of Gabriel and his use of the spectrum of indecomposable
injectives has been a great inspiration for non-commutative algebraic
geometry; see for example \cite{Br2018,Ka2015,Ro1995} where several possible
definitions of a spectrum of an abelian category are discussed. Here
we follow a somewhat different route. For not necessarily commutative
rings there is a spectrum which receives much attention and provides
another motivation for this work.

The \emph{Ziegler spectrum} of a ring $\La$ is given by the
isomorphism classes of indecomposable pure-injective $\La$-modules,
together with a topology introduced by Ziegler in model-theoretic
terms \cite{Zi1984}. Let $\Ind\La$  denote the Ziegler spectrum
of $\La$. This identifies with the spectrum of a locally coherent
Grothendieck category; so the above theory applies. To be more
precise, the module category $\Mod\La$ embeds into its purity category
$\bfP(\Mod\La)$, and then $\Ind\La$ identifies with
$\Sp\bfP(\Mod\La)$. The Boolean spectrum provides an extension
\begin{equation}\label{eq:Ind-PSpc}
  \Ind\La\iso\Sp\bfP(\Mod\La)\hookrightarrow\Spc\bfP(\Mod\La)=:
  \PSpc\La
\end{equation}
and we propose to study this as an invariant of $\La$.

The Ziegler spectrum has been used in various contexts, for instance
when studying representations of Artin algebras, but also in more
general categorical settings, including derived and stable categories;
see \cite{Pr2009} for an encyclopedic survey. An important feature of
this spectrum is Crawley-Boevey's fundamental correspondence
\[\C\longmapsto\C\cap\Ind\La\]
between definable subcategories of $\Mod\La$ and Ziegler closed
subsets of $\Ind\La$; see \cite{CB1998}.  In order to illustrate the
use of the Boolean spectrum $\PSpc\La$ we extend this correspondence
as follows.

\begin{Thm}[Theorem~\ref{th:pure-spectral}]
  The assignment $\C\mapsto\supp(\C)$ induces an inclusion preserving bijection between
  \begin{enumerate}
  \item the set of subcategories of $\Mod\La$ that are closed under arbitrary
    direct sums, pure subobjects, and pure-injective envelopes, and
  \item the clopen subsets of $\PSpc\La$.
  \end{enumerate}
  In particular, the embedding \eqref{eq:Ind-PSpc} identifies $\C\cap\Ind\La$
with  $\supp(\C)\cap\Ind\La$ when $\C$ is definable.  
\end{Thm}

Returning to the above notion of support $X\mapsto \supp(X)$ for an
arbitrary Grothendieck category, we mention that there is also an
exact version $X\mapsto\suppex(X)$ such that for each exact sequence
$0\to X_1\to X_2\to X_3\to 0$  one
has
\[\suppex(X_i)\subseteq\suppex(X_j)\cup\suppex(X_k)\qquad\text{when}\qquad\{i,j,k\}=\{1,2,3\}.\]
This is in line with Neeman's celebrated classification of localising
subcategories for the derived category of a commutative noetherian
ring \cite{Ne1992}, because it agrees with his notion of support for a
complex when specialised to a module. The exact support provides a
classification of certain thick subcategories
(Theorem~\ref{th:cohstable}), which is an analogue of the bijective
correspondence \eqref{eq:supp(C)} and again not requiring any
noetherianess assumption. This generalises a classification for
modules over a commutative noetherian ring from \cite{Ta2009} and
another more recent classification for locally noetherian categories
\cite{WM2024}.

Let us give a quick outline of this paper. In \S\ref{se:spec} we
recall the basic facts about spectral categories and introduce the
Boolean spectrum of a Grothendieck category. In \S\ref{se:spec-sub} we
set up the support for Grothendieck categories and establish the
classification of essentially closed subcategories. A first
application is given in \S\ref{se:decomp}, where coproduct
decompositions of objects are studied in terms of Boolean lattices.
The following \S\ref{se:coh-stable-sub} provides the derived analogue
of the results from \S\ref{se:spec-sub}, which leads to the exact
support. The final \S\ref{se:ex-definable} discusses the applications
for module categories, but in the more general setting of exactly
definable categories.

\section{Spectral Grothendieck categories}\label{se:spec}

Let $\A$ be a Grothendieck category. We recall the notion of its
spectral category from the work of Gabriel and Oberst
\cite{GO1966}. Then we study the lattice of localising subcategories
and use this to define a space, which we call the Boolean spectrum of
$\A$.

\subsection*{Spectral categories}
We write $\Spec\A$ for the \emph{spectral category} of $\A$ which is
obtained from $\A$ by formally inverting all essential monomorphisms;
see \cite{GO1966} or \cite[V.7]{St1975}.  The category $\Spec\A$ is
again a Grothendieck category in which all exact sequences are split
exact. A Grothendieck category with this property is called
\emph{spectral}.  The canonical functor
\[P\colon\A\lto\Spec\A:=\A[\Ess^{-1}]\] maps any injective envelope
$X\to E(X)$ in $\A$ to an isomorphism and it identifies the
indecomposable injective objects in $\A$ with the simple objects in
$\Spec\A$. These properties are clear from the definition. Also, $P$
is left exact and preserves all coproducts since the class of
essential monomorphisms is closed under all coproducts.
More precisely, for any directed union $X=\sum_\a X_\a$ in $\A$ we have  $P(X)=\sum_\a P(X_\a)$.
We write
$\Sp\A$ for the set of isomorphism classes of indecomposable injective
objects in $\A$, and $P$ induces a bijection $\Sp\A\iso\Sp(\Spec\A)$.

Let $\Inj\A$ denote the full subcategory of injective objects in
$\A$. For later use we note that $P$ induces an equivalence
\begin{equation}\label{eq:Spec}
  (\Inj\A)/\Rad(\Inj\A)\longiso\Spec\A,
\end{equation}
where $\Rad\C$ denotes the Jacobson radical of an additive category
$\C$; see \cite[Proposition~2.5.9]{Kr2022}. Recall that $\Rad\C$ is
an ideal of morphisms that is given by subgroups
\[\Rad(X,Y)\subseteq\Hom(X,Y)\] for each pair of objects $X,Y$ such that
$\Rad(X,X)$ equals the Jacobson radical of the endomorphism ring
$\End(X)$. A useful consequence of \eqref{eq:Spec} is that for any
pair of objects $X,Y\in\A$ we have
\begin{equation}\label{eq:Spec-iso}
 P(X)\cong P(Y)\qquad\iff\qquad  E(X)\cong E(Y)
\end{equation}
where $E(X)$ denotes an injective envelope of $X$; see also
\cite[Satz~1.5]{GO1966}.

The following lemma shows that the spectral category controls the
decompositions of injective objects.

\begin{lem}\label{le:spec-decomp-inj}
  Let $X\in\A$ be an injective object and $P(X)=Y_1\oplus Y_2$ a
  decomposition. Then there exists a decomposition $X=X_1\oplus X_2$
  such that $Y_i=P(X_i)$ for $i=1,2$.
\end{lem}
\begin{proof}
There is a decomposition into idempotent endomorphisms
$\id_{P(X)}=\p_1+\p_2$ such that $Y_i=\Ker\p_i$. Using \eqref{eq:Spec}
we decompose $\id_X=\psi_1+\psi_2$ such that $P(\psi_i)=\p_i$. Now set
$X_i=E(\Ker\psi_i)\subseteq X$ and apply \eqref{eq:Spec-iso}.
\end{proof}

Spectral categories and their structure are closely related to
(modules over) self-injective von Neumann regular rings, because the
spectral categories are precisely (up to equivalence) the categories
of non-singular injective modules over such rings.  This is explained
in \cite{GO1966,Ro1967}, and for the structure of such rings we refer
to \cite{Go1979}.\footnote{Given a spectral category and a
  generator $G$, the endomorphism ring of $G$ is self-injective von
  Neumann regular, and the equivalence is induced by $\Hom(G,-)$ \cite[Satz~2.1]{GO1966}.}
There is an elaborated decomposition theory for operator algebras
which uses certain types and goes back to Murray, von Neumann, and
Kaplansky \cite{Ka1951,MvN1936}. The analogous decomposition theory
for spectral categories is developed in \cite{GB1976,Ro1967}. For some
further structure of spectral categories, see \cite{Fa1983}.

\subsection*{Localising subcategories}

Let $\A$ be a Grothendieck category.  A full subcategory of $\A$ is
\emph{localising} if it is closed under subobjects, quotients,
extensions, and arbitrary coproducts. The localising subcategories of
$\A$ form a set, because for a fixed generator $G\in\A$ any localising
subcategory $\U\subseteq\A$ is determined by the set of quotients
$G/G'$ which belong to $\U$. The localising subcategories are
partially ordered by inclusion and closed under arbitrary
intersections; so they form a complete lattice which we denote by
$\bfL(\A)$.

For a localising subcategory $\U\subseteq\A$ and an object $X\in\A$ let
$t_\U(X)\subseteq X$ denote the maximal subobject in $\U$.  The
following lemma provides a way of describing the join in $\bfL(\A)$.

\begin{lem}\label{le:join}
Let $(\U_\a)$ be a family in
$\bfL(\A)$ and set $\U=\bigvee_\a\U_\a$. For any $X\in\A$ we have
\[\sum_\a t_{\U_\a}(X)=t_\U(X).\]
In particular,
\[t_\U(X)=0\qquad\iff\qquad t_{\U_\a}(X)=0\quad \text{for
    all}\quad\a.\]
\end{lem}
\begin{proof}
  We begin with the second assertion.  Let $X\to E(X)$ denote an injective
  envelope. Then $t_\U(X)=0$ if and only if $t_\U(E(X))=0$. So
  we may assume that $X$ is injective. But then the assertion is
  clear, since $t_\U(X)=0$ if and only if $\Hom(U,X)=0$ for all
  $U\in\U$, and the kernel of $\Hom(-,X)$ is a localising subcategory.

  Now consider the subobjects
  \[X':=\sum_\a t_{\U_\a}(X) \subseteq t_\U(X) \subseteq X\] If $X$ is
  in $\U$, then $t_{\U_\a}(X/X')=0$ for all $\a$, and therefore
the first part of the proof yields  $t_\U(X/X')=0$. Thus $X'=X$.
\end{proof}

A complete lattice is a \emph{frame} if finite meets distribute over
arbitrary joins, and it is a \emph{coframe} if finite joins distribute
over arbitrary meets. It is well known that the localising
subcategories of a Grothendieck form a frame; see for
instance \cite{BK1987} (in a more general categorical context) or
\cite{Go1986} (for module categories). We refer to \cite{St2024} for a
proof that uses an explicit construction of the join in the lattice of
localising subcategories.

\begin{prop}\label{pr:lattice-loc}
  Let $\A$ be a Grothendieck category. For an element $\U$ and a family $(\V_\a)$ in
  $\bfL(\A)$ we have
  \[\U\wedge\left(\bigvee_\a\V_\a\right)=\bigvee_\a(\U\wedge\V_\a).\]
  Thus the lattice $\bfL(\A)$ is a frame.
\end{prop}
\begin{proof}
  The inclusion $\supseteq$ is automatic; so we need to show
  $\subseteq$. Let $X$ be an object in
  $\U\wedge\left(\bigvee_\a\V_\a\right)$. Then Lemma~\ref{le:join}
  implies $X=\sum_\a t_{\V_\a}(X)$, and therefore 
$X$ belongs to  $\bigvee_\a(\U\wedge\V_\a)$ since $
t_{\V_\a}(X)\in \U\wedge\V_\a$ for each $\a$.
\end{proof}

\begin{prop}\label{pr:lattice-loc-spec}
  Let $\A$ be a spectral Grothendieck category. For an element $\U$ and a family $(\V_\a)$ in
  $\bfL(\A)$ we have
  \[
    \U\vee\left(\bigwedge_\a\V_\a\right)=\bigwedge_\a(\U\vee\V_\a).\]
  Thus the lattice $\bfL(\A)$ is a coframe.
\end{prop}
\begin{proof}
  The inclusion $\subseteq$ is automatic; so we need to show
  $\supseteq$. Any object $X$ can be written as
  $X=t_\U(X)\oplus X/t_\U(X)$. A simple calculation shows that
  $X\in\U\vee\V_\a$ if and only if $X/t_\U(X)\in\V_\a$. Thus
  $X\in\bigwedge_\a(\U\vee\V_\a)$ implies
  $X/t_\U(X)\in\bigwedge_\a\V_\a$, and therefore
  $X\in \U\vee\left(\bigwedge_\a\V_\a\right)$.
\end{proof}

\subsection*{Discrete versus continuous}

Let $\A$ be a spectral category. The category $\A$ is called
\emph{discrete} if every object is a coproduct of simple objects. Any
object $X\in\A$ admits a decomposition \[X=\soc(X)\oplus\rad(X),\] where the
\emph{socle} $\soc(X)$ is the sum of all simple subobjects and the
\emph{radical} $\rad(X)$ is the intersection of all maximal
subobjects. We set
\[\A_{\mathrm d}:=\{X\in\A\mid\soc(X)=X\}\qquad\text{and}\qquad
  \A_{\mathrm c}:=\{X\in\A\mid\rad(X)=X\}.\] This yields a
decomposition $\A=\A_{\mathrm d}\times \A_{\mathrm c}$ of spectral
Grothendieck categories such that $\A_{\mathrm d}$ is discrete and
$(\A_{\mathrm c})_{\mathrm d}=0=(\A_{\mathrm d})_{\mathrm c}$; see  \cite[\S3]{GO1966}.

\begin{lem}\label{le:discrete}
  Let $\A$ be spectral. The assignment $\U\mapsto\U\cap\Sp\A$ induces
  a lattice isomorphism between $\bfL(\A_{\mathrm d})$ and the lattice
  of subsets of $\Sp\A$.
\end{lem}
\begin{proof}
  The inverse map sends a subset $\V\subseteq\Sp\A$ to the subcategory
  of $\A$ consisting of all coproducts of objects in $\V$.
\end{proof}

\begin{exm}\label{ex:noeth}
  Let $\A$ be a locally noetherian Grothendieck category.  Then its
  spectral category $\Spec\A$ is discrete \cite[IV.2]{Ga1962}. Thus in
  this case $\bfL(\Spec\A)$ identifies with the lattice of subsets of
  $\Sp\A$.
\end{exm}

\subsection*{The Boolean spectrum}

Let $\A$ be a Grothendieck category. For a subcategory $\C\subseteq\A$ we set
\begin{align*}
  \C^\perp&:=\{X\in\A\mid\Hom(C,X)=0\text{ for all }C\in\C\},\\
^\perp\C&:=\{X\in\A\mid\Hom(X,C)=0\text{ for all }C\in\C\}.
\end{align*}

Localising subcategories of $\A$ are in bijective correspondence to
hereditary torsion pairs, by taking a subcategory $\C$ to the pair
$(\C,\C^\perp)$; see \cite[VI.3]{St1975}. Torsion pairs for spectral
categories admit the following elementary description. In particular,
this shows that any torsion class is a localising and colocalising subcategory.

\begin{lem}\label{le:torsion}
    Let $\A$ be a spectral Grothendieck category. A pair  $(\C,\D)$ of
    full subcategories is a torsion pair if and only if
    \begin{enumerate}
      \item $\C$ and $\D$ are closed under direct summands,
      \item $\C\cap\D=0$, and
        \item $\A=\{C\oplus D\mid C\in\C,\, D\in\D\}$.
\end{enumerate}
\end{lem}
\begin{proof}
Clearly, for any  torsion pair $(\C,\D)$ the conditions (1)--(3) hold. Conversely,
these conditions  imply that $\C^\perp=\D$ and $\C={^\perp\D}$. Thus  $(\C,\D)$ is a torsion pair.
\end{proof}

We record the above in lattice-theoretical terms. Recall that a
bounded lattice $L$ is \emph{complemented} if for each $x\in L$ there
exists a \emph{complement} $y\in L$, so $x\wedge y=0$ and $x\vee
y=1$. Such a complement is unique when $L$ is distributive, and then
we write $\neg x$ for the complement of $x$. A
complemented distributive lattice is called \emph{Boolean lattice}.

\begin{prop}\label{pr:Boole}
  Let $\A$ be a spectral Grothendieck category.  Then the lattice
  $\bfL(\A)$ is a complete Boolean lattice.
\end{prop}
\begin{proof}
  First observe that all torsion pairs are hereditary since $\A$ is
  spectral.  Now let $(\C,\D)$ be a torsion pair. Then it follows from
  Lemma~\ref{le:torsion} that $(\D,\C)$ is again a torsion pair. Thus
  $\D=\C^\perp$ is a complement for $\C\in\bfL(\A)$.
\end{proof}

Let us collect some basic facts from Stone duality; for details and
proofs see \cite{DST2019,Jo1982}.  Let $A=(A,\vee, \wedge)$ be a
Boolean lattice. We may identify $A$ with the corresponding
\emph{Boolean ring} $(A,+,\cdot)$, with operations given by
\[x+y:=(x\wedge\neg y)\vee (y\wedge\neg x)\qquad\text{and}\qquad
  x\cdot y:=x\wedge y,\] and keeping in mind that the category of
Boolean lattices is isomorphic to the category of Boolean rings. Let
$\Spec(A)$ denote the spectrum of prime ideals of $A$ with the Zariski
topology. For a topological space $T$ we write $\Clop(T)$ for its
Boolean lattice of clopen subsets. This yields an adjoint pair of
contravariant functors
\[
  \begin{tikzcd}[column sep=large]
    \Top \arrow[rightarrow,yshift=.75ex,rrr,"{\Clop=\Hom(-,\two)}"]
    &&& \arrow[rightarrow,yshift=-.75ex,lll,"{\Spec=\Hom(-,\two)}"]
    \BRing
  \end{tikzcd}
\]
where $\two=\{0,1\}$ is either viewed as topological space with the
discrete topology, or $\two=\mathbb F_2$ is viewed as a Boolean
ring. This means there is a natural bijection
\begin{equation}\label{eq:stone}
 \Hom(A,\Clop(T))\iso\Hom(T,\Spec(A))
\end{equation}
which sends $\p\colon A\to\Clop(T)$ to the map $T\to\Spec(A)$ given by
\[p\longmapsto\{x\in A\mid p\not\in\p(x)\}\qquad\text{for }p\in T.\]
\emph{Stone duality} implies that for any Boolean ring $A$ the
unit of the adjunction
\begin{equation}\label{eq:stone-unit}
  A\lto \Clop(\Spec(A)),\quad x\mapsto\{\frp\in\Spec(A)\mid
  x\not\in\frp\}
\end{equation}  
is an isomorphism.

\begin{defn}
  The \emph{Boolean spectrum} of a Grothendieck category $\A$ is the space
  \[\Spc\A:=\Spec(\bfL(\Spec\A)).\]
\end{defn}

For a Grothendieck category $\A$ the decomposition
$\Spec\A=(\Spec\A)_{\mathrm d}\times (\Spec\A)_{\mathrm c}$ implies
that $\Spc\A$ equals the disjoint union
\begin{equation}\label{eq:Spec-decomposition}
  \Spc\A=\Spec(\bfL((\Spec\A)_{\mathrm d}))\sqcup
  \Spec(\bfL((\Spec\A)_{\mathrm c}))
\end{equation}
where
\[ \Spec(\bfL((\Spec\A)_{\mathrm d}))\cong\Sp\A\] carries the discrete
topology. In particular, we see that the set $\Sp\A$ of
isoclasses of indecomposable injective objects embeds canonically into $\Spc\A$; cf.\
Lemma~\ref{le:discrete}.

\subsection*{Central idempotents}

For an additive category $\A$ let $Z(\A)$ denote its \emph{centre}
which is the ring of natural transformations $\id_\A\to\id_\A$. For a
ring $\La$ we write $Z(\La)$ for its centre and let $\bfB(\La)$ denote
the lattice of central idempotents, with operations given by
  \[x\vee y :=x+y-x\cdot y \qquad\text{and}\qquad x\wedge y:=x\cdot
    y.\] 

Let $\A$ be a spectral Grothendieck category. Any torsion pair
$(\C,\D)$ gives rise to a functorial decomposition $X=t_\C(X)\oplus
t_\D(X)$ for each object $X\in\A$, which amounts to an idempotent
$e_\C\colon\id_\A\to\id_\A$ such that the image of $(e_\C)_X$ equals
$t_\C(X)$. Conversely, any idempotent $e\in Z(\A)$ yields a torsion
pair $(\C,\D)$ if we set $\C=\{X\in\A\mid \Im e_X=X\}$ and
$\D=\{X\in\A\mid \Ker e_X=X\}$. Clearly, these operations are mutually
inverse to each other, and this yields another description of
$\bfL(\A)$.

\begin{lem}\label{le:centre}
  Let $\A$ be a spectral Grothendieck category. Then the assignment
  $\C\mapsto e_\C$ induces a lattice isomorphism
  \[\bfL(\A)\longiso\bfB(Z(\A)).\]
  Moreover, when $G\in\A$ is a generator with $\La=\End(G)$, then
   \[x\longmapsto\{X\in\A\mid \Hom(G,X)x=\Hom(G,X)\}\] induces a lattice isomorphism
  \[\bfB(\La)\longiso\bfL(\A).\]
\end{lem}
\begin{proof}
  The first assertion is clear from the above discussion. Let $\A'$
  denote the category of non-singular injective right modules over
  $\La$.  The functor $\Hom(G,-)$ induces an equivalence $\A\iso \A'$
  and an isomorphism $Z(\A)\iso Z(\A')$. On the other hand,
  multiplication with a central element induces an isomorphism
  $Z(\La)\iso Z(\A')$.  Thus the second assertion follows from the
  first.
\end{proof}

We record the following for later use.

\begin{rem}\label{re:centre-regular}
  Let $\A$ be a spectral Grothendieck category and $G\in\A$ a
  generator with $\La=\End(G)$. Then $Z(\A)\cong Z(\La)$ and this is a
  self-injective von Neumann regular ring; see
  \cite[Theorem~2.1]{AS1974}.
\end{rem}

I am grateful to Ken Goodearl for pointing out the following converse
to Proposition~\ref{pr:Boole}.

\begin{prop}[Goodearl]
Every complete Boolean lattice is isomorphic to $\bfL(\A)$ for some
spectral Grothendieck category $\A$.
\end{prop}

\begin{proof}
  We begin with the following observation. Let $\La$ be a right
  self-injective von Neumann regular ring and let $\A$ denote the
  corresponding spectral category consisting of the non-singular
  injective $\La$-modules. Then it follows from
  Lemma~\ref{le:centre} that the assignment
  $x\mapsto\{M\in\A\mid Mx=M\}$ induces a lattice isomorphism
  $\bfB(\La)\iso\bfL(\A)$.

  Now let $L$ be a complete Boolean lattice and view this as a Boolean
  ring. Then $L$ is von Neumann regular with $\bfB(L)=L$, and the ring
  is self-injective because $L$ is complete
  \cite[Corollary~4]{BL1959}.  Thus for $\A$ the category of
  non-singular injective $L$-modules, we have that $L\cong \bfL(\A)$.
\end{proof}

The proof shows that the realisation of a complete Boolean lattice is
canonical.  For a ring $\La$ let $\Inj_{\mathrm{ns}}\La$ denote the
category of non-singular injective right modules over $\La$. Then for any complete
Boolean lattice $L$ there is a canonical isomorphism
\[L\longiso\bfL(\Inj_{\mathrm{ns}} L).\]

The above result suggests a procedure for constructing explicit
examples of objects with specific decomposition properties. For
instance, we obtain examples of superdecomposable objects. Recall that
an object in an abelian category is \emph{superdecomposable} if there
are no indecomposable direct summands.

\begin{exm}[Goodearl]\label{ex:superdecomp}
  Let $\La$ be an infinite direct product of copies of $\bbZ/2$ and
  let $\bar\La$ denote the maximal quotient ring of
  $\La/\soc(\La)$. Then $\bar\La$ is injective as a $\bar\La$-module and
  superdecomposable by construction, since $\soc(\bar\La)=0$.  Moreover,
  $\bfB(\bar\La)=\bar\La$ for its Boolean lattice of direct summands.

The ring $\bar\La$ has characteristic $2$ and is of type I, but
further examples of different characteristic and different type can be
constructed via methods from Propositions~5-3.11 and 5-3.12 in
\cite{GW2005}.
\end{exm}

For general Grothendieck categories the decompositions of objects will
be discussed in more detail in \S\ref{se:decomp}.

\subsection*{Coproduct decompositions}

We provide an explicit description of the join in the lattice
$\bfL(\A)$ when $\A$ is spectral. We use the following property of a
Boolean lattice.

\begin{lem}\label{le:BA-decomp}
  Let $A$ be a complete Boolean lattice and $x=\bigvee_\a x_\a$ an
  element. Then there are elements $y_\a\le x_\a$ such that
  $x=\bigvee_\a y_\a$ and $y_\a\wedge y_\b=0$ for all $\a\neq\b$.
\end{lem}
\begin{proof}
  We consider the set of families $(y_\a)$ of objects in $A$ such that  $y_\a\le x_\a$ and
  $y_\a\wedge y_\b=0$ for all $\a\neq\b$. This set is partially ordered
  via $(y_\a)\le (y'_\a)$ if $y_\a\le y'_\a$ for all $\a$. An
  application of Zorn's lemma yields a maximal element and it is
  easily checked that  $x=\bigvee_\a y_\a$ for such maximal  $(y_\a)$.
\end{proof}

\begin{prop}\label{pr:spec-join}
  Let $\A$ be a spectral Grothendieck category. For a family $(\U_\a)$ in
  $\bfL(\A)$ we have
  \[\bigvee_\a\U_\a=\left\{\coprod_\a X_\a\mid X_\a\in\U_\a\text{
    for all }\a\right\}.\]
\end{prop}
\begin{proof}
  Set $\U=\bigvee_\a\U_\a$. We appy Lemma~\ref{le:BA-decomp} and
  obtain a decomposition $\U=\bigvee_\a\V_\a$ such that
  $\V_\a\subseteq\U_\a$ and $\V_\a\cap\V_\b=0$ for all
  $\a\neq\b$. Given $X\in\U$ we have $X=\sum_\a t_{\V_\a}(X)$ by
  Lemma~\ref{le:join}. Moreover,  $t_{\V_\b}(X)\cap\sum_{\a\neq \b}
  t_{\V_\a}(X)=0$ for all $\b$ since $\V_\b\wedge\left(\bigvee_{\a\neq\b}\V_\a\right)=0$.
Thus   $X=\coprod_\a t_{\V_\a}(X)$.
\end{proof}

\subsection*{Functoriality}

We  briefly discuss the functoriality of the assignment
$\A\mapsto\Spec\A$.  An exact and coproduct preserving functor
$\A\to \B$ between Grothendieck categories admits a right adjoint
which restricts to an additive functor $\Inj\B\to\Inj\A$.  Using
\eqref{eq:Spec} one obtains an exact functor $\Spec\B\to\Spec\A$, at
least when $\A\to \B$ is the inclusion of a localising subcategory or
a quotient functor.

Now let $\A'\subseteq\A$ be a localising subcategory and set
$\A''=\A/\A'$.

\begin{lem}
The exact sequence
$\A'\rightarrowtail\A\twoheadrightarrow\A''$ induces an exact sequence
\[\Spec \A''\rightarrowtail\Spec\A\twoheadrightarrow\Spec\A'.\]
These functors admit right adjoints and yield  decompositions
\[\Spec\A=\Spec\A'\times\Spec\A''\qquad\text{and}\qquad \Spc\A=\Spc\A'\sqcup\Spc\A''.\]
\end{lem}
\begin{proof}
  The functor $\Spec\A\to\Spec\A'$ is induced by
  $t_{\A'}\colon\A\to\A'$. It is easily checked that
  $\Spec \A''\to\Spec\A$ identifies $\Spec \A''$ with the kernel of
  $\Spec\A\to\Spec\A'$, which is closed under arbitrary
  coproducts. The inclusion $\A'\to\A$ preserves essential extensions
  and induces therefore a functor $\Spec\A'\to\Spec\A$ which is a
  right inverse for $\Spec\A\to\Spec\A'$. Thus $\Spec\A\to\Spec\A'$ is
  a localisation functor, and this yields the decomposition
  $\Spec\A=\Spec\A'\times\Spec\A''$. The functor from spectral
  categories to spaces given by $\C\mapsto\Spec\bfL(\C)$ is
  contravariant and takes products to coproducts.
\end{proof}

\subsection*{Monoidal structure}

Let $\A$ be a spectral Grothendieck category. Following the suggestion
of an anonymous referee (who generously provided all you
need\footnote{2 Corinthians 9:8}) we discuss briefly a symmetric
monoidal structure on a closely related spectral Grothendieck category
$\A'$. In fact, the construction of $\A'$ is canonical: The centre $Z:=Z(\A)$ is known to be a self-injective
von Neumann regular ring, which can also be computed as $Z(\End(G))$
for any generator $G\in\A$; see
Remark~\ref{re:centre-regular}. Therefore $\Hom_\A(-,-)$ takes
naturally values in the Giraud subcategory $\A'\subseteq\Mod Z$
consisting of the non-singular injective $Z$-modules, which is
equivalent to $(\Mod Z)/({^\perp Z})$ and inherits a tensor product
from $\Mod Z$. Thus any choice of a generator $G\in\A$ yields a
functor $-\otimes_G-\colon\A\times \A\to\A'$ given by
\[X\otimes_G Y := \Hom(G,X)\otimes_Z\Hom(G,Y).\]

The functor $\Hom(G,-)\colon\A\to\A'$ induces an isomorphism
$\lambda\colon\bfL(\A')\iso\bfL(\A)$, with inverse induced by its left
adjoint $-\otimes_ZG$.  Here, one uses that localising and
colocalising subcategories agree for spectral Grothendieck categories.
Note that $\lambda$ does not depend on the
choice of $G$; it equals the isomorphism
\[\bfL(\A')\cong \bfB(Z)\cong\bfL(\A)\]
from Lemma~\ref{le:centre}.

For an object $X$ let $\langle X\rangle$ denote the
localising subcategory generated by  $X$. Viewing
$\lambda$ as an identification we obtain the following.

\begin{lem}\label{le:tensor}
  For a pair of objects $X,Y$ in $\A$ we have \[\langle X\rangle\cap
  \langle Y\rangle=\langle X\otimes_GY\rangle.\]
\end{lem}
\begin{proof}
  Given the identification $\langle X\rangle=\langle
  \Hom(G,X)\rangle$, it suffices to show for $Z$-modules $M,N$ in
  $\A'$ that 
\[\langle M\rangle\cap
  \langle N\rangle=\langle M\otimes_ZN\rangle.\] 
For $A\in\A'$  we have $t_{\langle M\rangle}(A)=t_{\langle M\rangle}(Z)\otimes_Z A$.
  Thus for $A\in \langle M\rangle\cap \langle N\rangle$ we
  get
  \[A\cong t_{\langle M\rangle}(Z)\otimes_Z t_{\langle
      N\rangle}(Z)\otimes_Z A\in \langle t_{\langle
      M\rangle}(Z)\otimes_Z t_{\langle N\rangle}(Z)\rangle\subseteq
    \langle M\otimes_ZN\rangle.\] The other inclusion is clear.
\end{proof}

\section{Essentially closed subcategories}\label{se:spec-sub}

We fix a Grothendieck category $\A$ and introduce a class of
subcategories which is controlled by its spectral category.  This
yields a notion of support for the objects of $\A$.  To simplify
notation we set
\[\bfS(\A):=\bfL(\Spec\A)\iso\Clop(\Spc\A).\]

\subsection*{Essentially closed subcategories and support}

We begin with the main definition of this section; the terminology is
justified by Theorem~\ref{th:spectral}, because we look at the
operations that are preserved by the functor $P\colon\A\to\Spec\A$ inverting
all essential monomorphisms.

\begin{defn}\label{de:spectral-subcat}
  A full subcategory of $\A$ is called \emph{essentially closed} if it
  is closed under arbitrary coproducts, subobjects, and essential
  extensions.
\end{defn}

The essentially closed subcategories are partially ordered by inclusion and
closed under arbitrary intersections; so they form a complete
lattice. When $\A$ is spectral, then its essentially closed subcategories are
precisely its localising subcategories. For a class $\X\subseteq\A$ let
$\langle \X\rangle$ denote the smallest essentially closed subcategory of $\A$
that contains $\X$.

\begin{lem}\label{le:spectral}
    Let $\A$ be a Grothendieck category.
  \begin{enumerate}
\item  If $\U\subseteq\Spec\A$ is localising, then
  $P^{-1}(\U)\subseteq\A$ is  essentially closed.
\item If $\C\subseteq\A$ is essentially closed, then $P(\C)\subseteq\Spec\A$ is localising.
  \end{enumerate}
\end{lem}
\begin{proof}
  (1) Let $X\subseteq Y$ be a subobject in $\A$ with $P(Y)\in\U$. Then
  $P(X)\in\U$ since $P$ is left exact. Thus $P^{-1}(\U)$ is closed
  under subobjects. Analogously, $P^{-1}(\U)$ is closed under
  coproducts since $P$ preserves coproducts, and it is closed under
  essential extensions as $P$ maps those to isomorphisms.

  (2) First observe that the subcategory $P(\C)$ is \emph{replete}, so given a pair of
  objects $X,Y\in\A$ such that $P(X)\cong P(Y)$, then $X\in\C$ implies
  $Y\in\C$. This follows from \eqref{eq:Spec-iso}, because when $X$
  and $Y$ have isomorphic injective envelopes, then $X\in\C$ implies
  $Y\in\C$.  Next we show that $P(\C)$ is closed under coproducts and
  subobjects. Clearly, $P(\C)$ is closed under coproducts since $P$
  preserves coproducts.  Any subobject $Y\subseteq P(X)$ of an object
  in $P(\C)$ is actually a direct summand of $P(X)$, and we may
  assume that $X$ is injective.  Using the equivalence
  \eqref{eq:Spec}, we have $Y\cong \Ker P(\p)\cong P(\Ker\p)$ for some
  morphism $\p\colon X\to X$. Thus $P(\C)$ is closed under subobjects.
\end{proof}

For an object $X\in\A$ its \emph{support} is by definition
\begin{equation}\label{eq:supp}
  \supp(X):=\inf\{\U\in\bfS(\A)\mid P(X)\in \U\}=\langle P(X)\rangle.
\end{equation}
For a class $\X\subseteq\A$ we set
$\supp(\X):=\bigvee_{X\in\X}\supp(X)$ and note that
\begin{equation}\label{le:supp-class}
   \supp(\X)=\bigvee_{X\in\X}\langle P(X)\rangle=\langle
    P(\X)\rangle.
\end{equation}

\begin{thm}\label{th:spectral}
  Let $\A$ be a Grothendieck category. The assignment
  \[\C\longmapsto\supp(\C)=P(\C)\] induces a lattice isomorphism
\[\{\C\subseteq\A\mid \C\text{ essentially closed}\}\longiso\bfS(\A).\]
The inverse map is given by  $\U\mapsto P^{-1}(\U)$.
\end{thm}

\begin{proof}
  The functor $P$ equals the identity on objects. From this and
  Lemma~\ref{le:spectral} it follows that $\C\mapsto P(\C)$ and
  $\U\mapsto P^{-1}(\U)$ provide mutually inverse bijections between
  the essentially closed subcategories of $\A$ and $\Spec\A$, respectively. The
  equality $\supp(\C)=P(\C)$ follows from \eqref{le:supp-class}.
\end{proof}

In this generality the theorem seems to be new.  Special cases were
studied before; see the discussion in Remark~\ref{re:cohstable} when
we establish another correspondence for subcategories of $\A$.
Essentially closed subcategories were introduced for module categories by Dauns
under the name `saturated class' \cite{Da1997}. He showed that they
form a Boolean lattice, though his proof is different from ours and
does not use the spectral category. Saturated classes were used in
\cite{Da1997} to develop a decomposition theory for modules, analogous
to the one for von Neumann regular rings. Proposition~\ref{pr:decomp}
gives some flavour of such decomposition results.

\begin{cor}\label{co:gen}
Every essentially closed subcategory of $\A$ is of the form $\langle X\rangle$
for some object $X\in\A$.
\end{cor}
\begin{proof}
  The assertion is clear when $\A$ is spectral, because any localising
  subcategory is a Grothendieck category and contains therefore a
  generator. When $\A$ is arbitrary and $\C\subseteq\A$ is essentially closed,
  we can choose $X\in\A$ with $P(\C)=\langle P(X)\rangle$. This implies $\C=\langle X\rangle$.
\end{proof}

\begin{rem}
  The complement of an essentially closed subcategory $\C\subseteq\A$ admits an
  easy description: it is the essentially closed subcategory consisting of all
  objects  $X\in\A$ such that any subobject of $X$ is zero when it
  belongs to $\C$.
\end{rem}

\subsection*{Support data}

The notion of a support datum for tensor triangulated categories was
introduced by Balmer in \cite{Ba2005}. The following is the analogue
for Grothendieck categories, where the intersection of essentially closed
subcategories plays the role of the tensor product.\footnote{This is
  reasonable since $\langle X\otimes Y\rangle=\langle X\rangle\cap
  \langle Y\rangle$ for objects $X,Y$ in a tensor triangulated
  category, where $\langle X\rangle$ denotes the radical thick tensor
  ideal generated by $X$. See also Lemma~\ref{le:tensor} for the
case of a spectral category.}

\begin{defn}\label{de:support}
A \emph{support datum} on a Grothendieck category $\A$ is a pair
$(T,\s)$ consisting of a topological space $T$ and a map
$\s\colon\Ob\A\to \Clop(T)$ such that
\begin{enumerate}
\item[(S$0$)] $\s(X)=\varnothing$ for $X=0$,
\item[(S$1$)] $\s(X)=T$ for every generator $X$ of $\A$,
\item[(S$\vee$)] $\s(X)\cup\s(Y)=\s(X\oplus Y)$ for all objects $X,Y$ in
  $\A$, and
\item[(S$\wedge$)] $\s(X)\cap\s(Y)=\s(Z)$ for all objects $X,Y,Z$ in $\A$
  with $\langle X\rangle\cap \langle Y\rangle=\langle Z\rangle$.
\end{enumerate}
A morphism $(T,\s)\to (T',\s')$ is given by a continuous map $f\colon
T\to T'$ such that $\s(X)=f^{-1}(\s'(X))$ for all $X\in\A$.
\end{defn}

It is convenient to reformulate the notion of a support datum as follows.
\begin{lem}\label{le:support-hom}
Let $(T,\s)$ be a support datum. Given objects $X,Y$ in $\A$, the
condition  $\langle X\rangle\subseteq\langle Y\rangle$ implies $\s(X)\subseteq
\s(Y)$.
Therefore the assignment $\langle X\rangle\mapsto\s(X)$ induces a map
\[\{\C\subseteq\A\mid \C\text{ essentially closed}\}\lto\Clop(T)\]
which is a lattice homomorphism.
\end{lem}
\begin{proof}
  From condition (S$\wedge$) it follows that
  $\langle X\rangle\subseteq\langle Y\rangle$ implies
  $\s(X)\subseteq \s(Y)$.  From Corollary~\ref{co:gen} we know that
  every essentially closed subcategory of $\A$ is generated by a single
  object. Thus the assignment $\langle X\rangle\mapsto\s(X)$ yields a
  well defined map
  $\{\C\subseteq\A\mid \C\text{ essentially closed}\}\to\Clop(T)$.  Conditions
  (S$0$)--(S$\wedge$) imply that this is a lattice homomorphism.
\end{proof}

By slight abuse of notation and using the isomorphism
\eqref{eq:stone-unit} we write $\supp$ for the composite
\[\Ob\A\xto{\ \supp\ }\bfS(\A)\longiso\Clop(\Spec(\bfS(\A)))=\Clop(\Spc\A).\]

\begin{thm}\label{th:support}
  The pair $(\Spc\A,\supp)$ is a support datum on $\A$. Moreover,
  for any support datum $(T,\s)$ there is a unique continuous map
  $f\colon T\to\Spc\A$ such that $\s(X)=f^{-1}(\supp(X))$ for all
  $X\in\A$. Explicitly, the map $f$ is given by the adjunction
  \eqref{eq:stone} and therefore
  \[f(p)=\{\langle X\rangle\in\bfS(\A)\mid p\not\in\s(X)\}\qquad\text{for
    }p\in T.\] 
\end{thm}

\begin{proof}
  It is clear that $(\Spc\A,\supp)$ is a support datum. Each support
  datum $(T,\s)$ induces a homomorphism
  $\{\C\subseteq\A\mid \C\text{ essentially closed}\}\to\Clop(T)$ by
  Lemma~\ref{le:support-hom}. The composite with the isomorphism
  $\bfS(\A)\iso\{\C\subseteq\A\mid \C\text{ essentially closed}\}$ from
  Theorem~\ref{th:spectral} yields a lattice homomorphism,
  which corresponds via the adjunction \eqref{eq:stone} to a
  continuous map $ T\to\Spec(\bfS(\A))=\Spc\A$ with the above
  explicit description.
\end{proof}

The clopen subsets of $\Spc\A$ form a complete lattice. This reflects
the existence of arbitrary coproducts in $\A$, but it need not hold
for an arbitrary space.

\begin{prop}
 Let $(X_\a)$ be a family of objects in $\A$. Then we have
 \[\supp\left(\coprod_\a X_\a\right)= \bigvee_\a\supp(X_\a).\]
\end{prop}

\begin{proof}
  The assertion follows from the equality
  $\langle\coprod_\a X_\a\rangle=\bigvee_\a\langle X_\a\rangle$.
\end{proof}

We end our discussion of support with a simple observation which
demonstrates the benefit of the extension $\Sp\A\hookrightarrow\Spc\A$.

\begin{rem}
For $X\in\A$ we have $\supp(X)\neq\varnothing$ when $X\neq 0$. This
property 
cannot be expected for a map $\Ob\A\to\Sp\A$ since $\Sp\A$ may be empty.
\end{rem}

\subsection*{Localising subcategories}

Next we show that the functor $\A\to\Spec\A$ induces an embedding for
the frame of localising subcategories. To this end we consider
essentially closed subcategories that are closed under arbitrary products
because they arise naturally from localising subcategories. The
following lemma explains a tight connection.
  
\begin{lem}\label{le:tf1}
    For a full subcategory $\D\subseteq\A$ the following are
    equivalent.
\begin{enumerate}
\item The subcategory $\D$ is essentially closed and closed under arbitrary
  products.
\item There is a localising subcategory $\C\subseteq\A$ such that
  $\C^\perp=\D$.
\end{enumerate}
\end{lem}

\begin{proof}
  (1) $\Rightarrow$ (2): We set $\C:={^\perp\D}$ and this is
  localising since $\D$ is closed under injective envelopes, so
  $\C={^\perp(\D\cap\Inj\A)}$. It remains to note that $\C^\perp=\D$
  where one uses that $\D$ is closed under arbitrary products; see
  \cite[Proposition~VI.2.2]{St1975}.
  
(2) $\Rightarrow$ (1): This is clear.
\end{proof}

We continue with a sequence of technical lemmas. 

\begin{lem}\label{le:tf3}
  Let $(\C_\a)$ be a finite family of essentially closed subcategories and
  $X\in\A$ an injective object. Then $X$ belongs to
  $\bigvee_\a\C_\a$ if and only if there is a decomposition
  $X=\coprod_\a X_\a$ such that $X_\a\in\C_\a$ for all $\a$.
\end{lem}

\begin{proof}
Consider $\C=\C_1\vee\C_2$.  If $X_1\in\C_1$ and $X_2\in\C_2$, then
  $X_1\oplus X_2\in\C_1\vee\C_2$.  Now assume $X\in\C_1\vee\C_2$. Set
  $Y=P(X)$ and $Y_1=t_{P(\C_1)}(Y)$. Then $Y_2=Y/Y_1$ is in
  $P(\C_2)$, and therefore $Y=Y_1\oplus Y_2$ with $Y_i\in P(\C_i)$. Using
  Lemma~\ref{le:spec-decomp-inj} we find $X_i\in\C_i$ such that
  $X=X_1\oplus X_2$.  See also Proposition~\ref{pr:decomp}.
\end{proof}

\begin{lem}\label{le:tf2}
  Let $(\C_\a)$ be a family of essentially closed subcategories and suppose that
  each $\C_\a$ is closed
  under arbitrary products. Then  $\bigwedge_\a\C_\a$ is closed under arbitrary
  products, and  $\bigvee_\a\C_\a$ is closed under arbitrary
  products provided the family is finite.
\end{lem}

\begin{proof}
  For the lattice of essentially closed subcategories the meet is given by
  taking intersections, and it is clear that being closed under
  products is preserved when taking arbitrary intersections. Now
  suppose the family $(\C_\a)$ is finite. Observe that an essentially closed
  subcategory $\C\subseteq\A$ is closed under products if and only if
  $\C\cap\Inj\A$ is closed under products.  For $\C=\bigvee_\a\C_\a$
  we note that $\C\cap\Inj\A$ consists of all objects
  $\coprod_\a X_\a$ with $X_\a\in\C_\a\cap\Inj\A$ for all $\a$, by
  Lemma~\ref{le:tf3}. For an arbitrary family of objects $(X_i)$ in
  $\C\cap\Inj\A$, the product
  \[\prod_i X_i=\prod_i \left(\coprod_\a X_{i,\a}\right)\cong \coprod_\a \left(\prod_i X_{i,\a}\right)\]
  belongs to $\C$ since $\prod_i X_{i,\a}$ belongs to $\C_\a$ for all
  $\a$.
\end{proof}

\begin{lem}\label{le:loc-inj}
  Let $\C\subseteq\A$ be localising. Then every injective object
  $X\in\A$ admits a decomposition $X=X'\oplus X''$ 
such that $t_\C(X)\to X'$ is an injective envelope and $t_\C(X'')=0$.
\end{lem}
\begin{proof}
  The inclusion $t_\C(X)\to X$ induces a monomorphism $E(t_\C(X))\to X$.
One takes for $X'$ its image and for $X''$ its cokernel.
\end{proof}

\begin{lem}\label{le:torsion-pair}
  Let $\C\subseteq\A$ be localising. Then the essential extensions of
  objects in $\C$ form an essentially closed subcategory $\bar\C\subseteq\A$, and
  the assignment
  \[(\C,\D)\mapsto (P(\bar\C),P(\D))\]
  yields a map between the hereditary torsion pairs for $\A$ and
  $\Spec\A$, respectively.
\end{lem}
\begin{proof}
  It is easily checked that the essential extensions of objects in
  $\C$ form an essentially closed subcategory. Now let $(\C,\D)$ be a hereditary
  torsion pair for $\A$. From Lemmas~\ref{le:torsion} and \ref{le:loc-inj}
  it follows that $(P(\bar\C),P(\D))$ is a  hereditary
  torsion pair for $\Spec\A$. 
\end{proof}

\begin{thm}\label{th:spectral-localising}
  Let $\A$ be a Grothendieck category. There is a pair of maps

  \[\begin{tikzcd}
\bfL(\A)  \arrow[rightarrow,yshift=.75ex,rr,"\s"]
 	&& \arrow[rightarrow,yshift=-.75ex,ll,"\tau"] \bfL(\Spec\A)
\end{tikzcd}\]
given by $\s(\C)=\supp(\C)$ and $\tau(\U)={^\perp (P^{-1}(\U^\perp))}$
such that $\t\s=\id$. Moreover, $\s$ preserves arbitrary joins and finite meets. 
\end{thm}
\begin{proof}
  First observe that $\t$ is well defined, since $P^{-1}(\U^\perp)$ is
  an essentially closed subcategory, and therefore ${^\perp(P^{-1}(\U^\perp))}$
  is a localising subcategory of $\A$.

  Let $\C\subseteq\A$ be localising and $(\C,\D)$ the corresponding
  hereditary torsion pair. Then $ (P(\bar\C),P(\D))$ is a hereditary
  torsion pair for $\Spec\A$ by Lemma~\ref{le:torsion-pair}, with
  $\U:=\s(\C)=P(\bar\C)$. Thus $\U^\perp=P(\D)$ and then
  $P^{-1}(\U^\perp)=P^{-1}P(\D)=\D$. It follows that $\t(\U)=\C$.

  The assignment $\C\mapsto\C^\perp$ identifies the localising
  subcategories of $\A$ with the essentially closed subcategories that are
  closed under products, by Lemma~\ref{le:tf1}. This map reverses
  inclusions and we compose it with  the isomorphism from
  Theorem~\ref{th:spectral}. Then it
  follows from Lemma~\ref{le:tf2} that the image of $\s$ is closed
  under  arbitrary joins and finite meets. 
\end{proof}

I am grateful to Sira Gratz and Greg Stevenson for suggesting the
following couple of remarks.

Given a frame $L$, its \emph{points} are the frame morphisms
$L\to\two$, where $\two$ denotes the frame with two elements
$\{0\le 1\}$.  The frame $L$ has \emph{enough points} if for all
$x\not\le y$ in $L$ there exists a point $p\colon L\to\two$ such that
$p(x)=1$ and $p(y)=0$.

\begin{rem}
  There are examples showing that the frame $\bfL(\A)$ may have no
  points \cite{St2024}. This is in contrast to
  Proposition~\ref{pr:Boole} which shows that the bigger frame
  $\bfL(\Spec\A)$ has enough points.
\end{rem}

\begin{rem}
  Suppose that $\Spec\A$ is discrete. For example, let $\A$ be locally
  noetherian. Then $\bfL(\A)$ has enough points, so it is a spatial
  frame. This is easily seen, because $\bfL(\Spec\A)$ identifies with the
  set of subsets of $\Sp\A$. For $I\in\Sp\A$ consider the point
  $p_I\colon\bfL(\A)\xto{\s}\bfL(\Spec\A)\xto{\pi_I}\two$, where
  $\pi_I(\U)=1$ iff $I\in \U$. Given a pair of localising subcategories
  $\C\not\subseteq\D$ of $\A$, one chooses
  $I\in\Sp\A\cap(\langle\C\rangle\smallsetminus\langle\D\rangle)$ and then $p_I(\C)=1$
  and $p_I(\D)=0$. 
\end{rem}

\section{Coproduct decompositions}\label{se:decomp}

Decompositions of objects in a Grothendieck category have been studied in great
detail in the past. For instance, the theorems of Azumaya, Krull,
Remak, and Schmidt treat the decompositions into indecomposable
objects with local endomorphism rings.\footnote{Splittings or
  decompositions have been considered since ancient times. The Quran
  reports about a miraculous splitting of the moon [Quran~54:1-2], and
  the Bible about the splitting of the Red Sea [Exodus~14:21-22].} For
arbitrary objects the decomposition theory is far more complex.  In
this section we discuss a Boolean lattice which controls the
decompositions, at least for injective objects. There is probably no
efficient way to describe all decompositions. So we focus on
decompositions up to a certain equivalence relation, which reflects
the decompositions of essentially closed subcategories.

We fix a Grothendieck category $\A$ and recall that $P\colon\A\to\Spec\A$
denotes the functor that inverts all essential monomorphisms.

\subsection*{Coproduct decompositions via essentially closed subcategories}

A coproduct decomposition of an object in $\A$ induces via the
assignment $X\mapsto \langle X\rangle$ a decomposition in the lattice
of essentially closed subcategories. This correspondence works in both
directions as follows.

\begin{prop} \label{pr:decomp}
  Let $\A$ be a Grothendieck category and $X\in\A$ an object.
  \begin{enumerate}
  \item A decomposition $X=\coprod_\a X_\a$ in $\A$ induces
  a decomposition $\langle X\rangle=\bigvee_\a\langle
  X_\a\rangle$ into  essentially closed subcategories.
\item Suppose that $X$ is injective and let
  $\langle X\rangle=\bigvee_\a\C_\a$ be a decomposition into essentially closed
  subcategories. Then there are direct summands $X_\a\subseteq X$ such
  that $X=E(\coprod_\a X_\a)$ and $X_\a\in\C_\a$ for all $\a$.
  Moreover, $\langle X_\a\rangle=\C_\a$ for all $\a$ provided that
  $\C_\a\wedge\C_\b=0$ for all $\a\neq\b$.
\end{enumerate}
\end{prop}
\begin{proof}
  (1) is clear from the definitions. For (2) we may assume that
  $\C_\a\wedge\C_\b=0$ for all $\a\neq\b$ by replacing $\C_\a$ with a
  smaller subcategory if needed; see Lemma~\ref{le:BA-decomp}.
  Suppose first that $\A$ is spectral. Then there is an essentially
  unique decomposition $X=\coprod_\a X_\a$ in $\A$ with
  $\langle X_\a\rangle=\C_\a$ for all $\a$, by choosing
  $X_\a=t_{\C_\a}(X)$. This follows from Lemma~\ref{le:join}; see also
  Proposition~\ref{pr:spec-join}. When $\A$ is arbitrary and $X$ is
  injective we obtain a decomposition of $X$ by decomposing $P(X)$ in
  $\Spec\A$ and applying Theorem~\ref{th:spectral}.  More precisely,
  for each $\a$ we apply Lemma~\ref{le:spec-decomp-inj} and choose a
  direct summand $X_\a\subseteq X$ such that
  $P(X_\a)=t_{P(\C_\a)}(P(X))$.  Then the equality
  $X=E(\coprod_\a X_\a)$ follows from \eqref{eq:Spec-iso}.
\end{proof}

\begin{rem}
  The identity $X=E(\coprod_\a X_\a)$ in Proposition~\ref{pr:decomp}
  simplifies to $X=\coprod_\a X_\a$ when the index set is finite or
  when $\A$ is locally noetherian.
\end{rem}

\begin{rem}
  In part (2) of the above proposition, the assumption on $X$ to be
  injective cannot be removed. Consider for instance an indecomposable
  object $X$ of length $3$ with $\soc(X)=S_1\oplus S_2$ such that
  $S_1$ and $S_2$ are non-isomorphic simple objects. Then
  $\langle X\rangle=\langle S_1\rangle\vee \langle S_2\rangle$ and
  $\langle S_1\rangle\wedge \langle S_2\rangle=0$, but there is no
  decomposition $X=X_1\oplus X_2$ with
  $\langle X_i\rangle=\langle S_i\rangle$.
\end{rem}

\subsection*{The lattice of decompositions}

We introduce an equivalence relation on the direct summands
of an injective object in order to provide an efficient description of
all decompositions.

\begin{defn}
  Let $X$ be an injective object in $\A$. Given a pair of direct summands $U\subseteq X$ and
  $V\subseteq X$, we say they are \emph{essentially equivalent} if
  $\langle U\rangle=\langle V\rangle$.  Let $\bfD(X)$ denote the set
  of essential equivalence classes of direct summands of $X$, with partial order
  given by $U\le V$ if $\langle U\rangle\subseteq\langle V\rangle$.
\end{defn}

The following result establishes a connection to the decomposition
theory for von Neumann regular rings since $\End(X)/J(\End(X))$ has
this property. In particular, the elaborated theory of types applies
\cite{Ka1951,MvN1936}.

\begin{thm}\label{th:decomp}
  Given an injective object $X$ in $\A$, the partially ordered set $\bfD(X)$ of direct
  summands is a complete Boolean lattice. Moreover, there is a
  canonical lattice isomorphism
\[\bfD(X)\longiso\bfB(\End(X)/J(\End(X))).\]
\end{thm}
\begin{proof}
We claim that the assignment $U\mapsto \langle P(U)\rangle= P(\langle
U\rangle)$ induces a lattice isomorphism
\begin{equation}\label{eq:decomp}
  \bfD(X)\longiso\{\U\in\bfS(\A)\mid \U\subseteq\supp(X)\}.
\end{equation}  
The inverse map sends a localising subcategory $\U\subseteq\Spec\A$ to
the equivalence class of a direct summand $U\subseteq X$ such that
$P(U)=t_\U(P(X))$; for the existence of $U$ see
Lemma~\ref{le:spec-decomp-inj}. Using Theorem~\ref{th:spectral} it is
easily checked that these assignments are mutually inverse to each
other and preserving the partial order. Thus $\bfD(X)$ is a Boolean
lattice, since $\bfS(\A)$ has this property by
Proposition~\ref{pr:Boole}. Moreover, the above argument shows that
$P\colon\A\to\Spec\A$ induces a lattice  isomorphism
\begin{equation}\label{eq:decomp1}
  \bfD(X)\longiso\bfD(P(X)).
\end{equation}

For the isomorphism with the lattice of central idempotents we
consider the Grothendieck category
\[\A_X:=\langle P(X)\rangle\subseteq\Spec\A\] which is generated by
$P(X)$. The functor $P$ induces an isomorphism
\[\End(X)/J(\End(X))\longiso \End(P(X))\] by \eqref{eq:Spec}. This
 yields the following sequence of isomorphisms
 \[\bfD(X)\iso\bfD(P(X))\iso\bfL(\A_X)\iso\bfB(\End(X)/J(\End(X))),\]
 where the first is \eqref{eq:decomp1}, the second is \eqref{eq:decomp}, and the last comes from
 Lemma~\ref{le:centre}.
\end{proof}

\begin{rem}
  The Boolean structure of $\bfD(X)$ implies that each equivalence class
  contains a summand $U\subseteq X$ that is maximal in the sense that
  $\langle U\rangle\cap \langle X/U\rangle=0$. An equivalent condition
  is that for any pair of monomorphisms $U\leftarrow V\to X/U$ we have
  $V=0$.
\end{rem}

\section{Cohomologically stable subcategories}\label{se:coh-stable-sub}

An essentially closed subcategory of a Grothendieck category is always closed
under subobjects and extensions, but it need not be closed under
quotients.  For that reason we consider another class of subcategories
which is also controlled by its spectral category. This leads to a
notion of exact support for any Grothendieck category.

\subsection*{Exact support}

Let $\A$ be an exact category. Thus $\A$ is an additive category
together with a distinguished class of short exact sequences. Recall that
a full subcategory $\C\subseteq\A $  is \emph{thick} if $\C$
is closed under direct summands and the \emph{two out of three
  condition} holds: any short exact sequence belongs to $\C$ if two of
its three terms are in $\C$.

A map $\s\colon \Ob\A\to L$ with
values in a join-semilattice $L$ is called \emph{exact} if
\begin{enumerate}
  \item for all objects $X,Y\in\A$ one has $\s(X\oplus Y)=\s(X)\vee\s(Y)$, and
  \item for each exact sequence $0\to X_1\to X_2\to X_3\to 0$ in $\A$
    one has
    \[\s(X_i)\le\s(X_j)\vee\s(X_k)\qquad\text{when}\qquad\{i,j,k\}=\{1,2,3\}.\]
  \end{enumerate}
  It is easily checked that $\s$ is exact if and only if for each
  $u\in L$ \[\A_u:=\{X\in\A\mid\s(X)\le u\}\] is a thick subcategory of $\A$.

  Now let $\A$ be a Grothendieck category. We consider again the
  Boolean lattice
\[\bfS(\A):=\bfL(\Spec\A)\iso\Clop(\Spc\A).\] Also, we consider the derived
  category $\bfD(\A)$ and the right derived functor
\[\bfR P\colon\bfD(\A)\lto\bfD(\Spec\A).\]
For an object $X\in\A$ its \emph{exact support} is by definition
\begin{equation}\label{eq:suppex}
  \suppex(X):=\inf\{\U\in\bfS(\A)\mid \bfR
  P(X)\in\bfD(\U)\}\in\bfS(\A),
\end{equation}  
where we identify
\[\bfD(\U)=\{Y\in\bfD(\Spec\A)\mid  H^n(Y)\in\U\text{ for all }n\in\bbZ\}.\]
For $\X\subseteq\A$ we set $\suppex(\X):=\bigvee_{X\in\X}\suppex(X)$.

\begin{lem}\label{le:res}
Let $X\in\A$ be an object and $X\to I$ a minimal injective resolution. Then
$\suppex(X)$ equals the localising subcategory of $\Spec\A$ consisting of all direct
summands of coproducts of objects in $\{P(I^n)\mid n\in\bbZ\}$.
\end{lem}
\begin{proof}
  The functor $P$ sends each differential $I^n\to I^{n+1}$ to zero
  because they are radical morphisms; see
  \cite[Proposition~4.3.18]{Kr2022}. Thus $H^n(\bfR P(X))\cong P(I^n)$
  for each $n\in\bbZ$.
\end{proof}

Let $\U\in\bfS(\A)$.  We set
\[\Inj_\U\A:=\{I\in\Inj\A\mid P(I)\in\U\}\]
and $\A_\U$ denotes the full subcategory of objects $X\in\A$ that
admit an injective resolution $X\to I$ with $I^n\in\Inj_\U\A$ for all
$n\in\bbZ$, so
\[\A_\U=\{X\in\A\mid \bfR
  P(X)\in\bfD(\U)\}=\{X\in\A\mid\suppex(X)\subseteq\U\}.\]

\begin{lem}\label{le:thick}
  For each $\U\in\bfS(\A)$ the subcategory $\A_\U$ is thick, and the
  function \[\suppex\colon\Ob\A\lto\bfS(\A)\] is exact.
\end{lem}

\begin{proof}
  Each exact sequence $0\to X\to Y\to Z\to 0$ in $\A$ gives rise to an
  exact triangle $X\to Y\to Z\to X[1]$ in $\bfD(\A)$. The right
  derived functor $\bfR P$ is exact and $\bfD(\U)$ is a triangulated
  subcategory which is closed under direct summands. Then it follows
  that $\A_\U$ is thick. We have already seen that thickness of
  $\A_\U$ for each $\U$ implies the exactness of the corresponding
  support function.
\end{proof}

\subsection*{Cohomologically stable subcategories}

The following definition is the derived analogue of Definition~\ref{de:spectral-subcat}.

\begin{defn}\label{de:stable-subcat}
  We call a thick subcategory $\C\subseteq\A$ 
  \emph{cohomologically stable} if
  \begin{enumerate}
\item for every family $(X_\a)$ of objects in $\C$ the injective
  envelope of  $\coprod_\a X_\a$ belongs to $\C$, and
\item for every exact sequence $X^0\to X^1\to X^2\to\cdots$ with all
  $X^n$ in $\C$ the kernel of $X^0\to X^1$ belongs to $\C$.
\end{enumerate}
\end{defn}

The cohomologically stable subcategories are partially ordered by
inclusion and closed under arbitrary intersections; so they form a
complete lattice. For example, the cohomologically stable
subcategories are precisely the localising subcategories when $\A$ is
spectral.  The fact that the cohomologically stable subcategories form a set
follows from Theorem~\ref{th:cohstable} below. For an object $X\in\A$
let $\langle X\rangle_\ex$ denote the smallest cohomologically stable
subcategory of $\A$ that contains $X$.

\begin{lem}\label{le:U}
  For $\U\in\bfS(\A)$ the subcategory $\A_\U$ is cohomologically stable.
\end{lem}
\begin{proof}
  The category $\A_\U$ is thick by Lemma~\ref{le:thick}.  Let $(X_\a)$
  be a family of objects in $\A_\U$ and choose injective envelopes
  $X_\a\to I_\a$. Then $I_\a\in\Inj_\U\A$ for all $\a$, and therefore
  $E(\coprod_\a X_\a)\cong E(\coprod_\a I_\a)$ belongs to
  $\Inj_\U\A$. Now let \[X^0\to X^1\to X^2\to\cdots\] be exact with all
  $X^n$ in $\A_\U$. We set $Z^n=\Ker (X^n\to X^{n+1})$ and need to
  show that $Z^0$ belongs to $\A_\U$. Choose an injective envelope
  $X^0\to I^0$. The cokernel $C=\Coker(Z^0\to I^0)$ fits into an exact
  sequence $0\to Z^1\to C\to C'\to 0$ with $C'=\Coker(X^0\to I^0)$ in
  $\A_\U$. We form the pushout
 \[\begin{tikzcd}
      0\arrow[r]&Z^1\arrow[r]\arrow[d]&X^1\arrow[r]\arrow[d]&
      Z^2\arrow[r]\arrow[d,equal]&0\\
   0\arrow[r]&C\arrow[r]&\bar X^1\arrow[r]&
      Z^2\arrow[r]&0
    \end{tikzcd}\]
and the object $\bar X^1$ belongs to $\A_\U$. Thus we obtain an exact sequence
 \[I^0\to \bar X^1\to X^2\to X^3\to \cdots\] with all
 terms in $\A_\U$ and $Z^0=\Ker (I^0\to\bar X^1)$.
Proceeding by induction we obtain 
 an exact sequence
 \[I^0\to I^1\to I^2\to I^3\to \cdots\] with all terms in $\Inj_\U\A$
 and $Z^0=\Ker (I^0\to I^1)$.  Thus $Z^0$ belongs to
 $\A_\U$.
\end{proof}

\begin{thm}\label{th:cohstable}
  Let $\A$ be a Grothendieck category. The assignment
  $\C\mapsto\suppex(\C)$ induces a lattice isomorphism
\[\{\C\subseteq\A\mid \C\text{ cohomologically stable}\}\longiso\bfS(\A).\]
The inverse map is given by  $\U\mapsto\A_\U$.
\end{thm}

\begin{proof}
  Lemma~\ref{le:U} shows that the assignment $\U\mapsto \A_\U$ is well
  defined. Using the description of support in Lemma~\ref{le:res}, it
  is clear that $\suppex(\A_\U)=\U$. Now let $\C\subseteq\A$ be a
  cohomologically stable subcategory and set $\U=\suppex(\C)$. We claim that
  $\Inj_\U\A\subseteq\C$. When $X\to I$ is a minimal
  injective resolution of an object in $X\in\C$, then the definition
  of cohomologically stable implies that $I^n\in\C$ for all
  $n\in\bbZ$. Moreover, cohomologically stable implies that any object
  $I\in\Inj\A$ belongs to $\C$ when $P(I)$ arises as a direct summand
  of a coproduct of objects of the form $P(I^n)$. Here we use again
  that $P$ preserves coproducts. This yields the claim, and the
  equality $\A_\U=\C$ then follows.
\end{proof}

\begin{rem}
  When $\A$ is locally noetherian, then a cohomologically stable
  subcategory is closed under all coproducts. This follows from the
  fact that any coproduct of injectives is again injective. Moreover,
  in this case $\bfS(\A)$ identifies with the lattice of subsets of
  $\Sp\A$; see Example~\ref{ex:noeth}.
\end{rem}

\begin{rem}\label{re:cohstable}
  In this generality Theorem~\ref{th:cohstable} seems to be new. For
  some special cases it was already known that subcategories of $\A$
  are determined by subsets of $\Sp\A$.

  (1) When $\A$ equals the category of modules over a commutative
  noetherian ring $\La$, then $\bfS(\A)$ identifies with the lattice
  of subsets of the prime ideal spectrum of $\La$; see
  Example~\ref{ex:noeth}. In this case the localising subcategories of
  $\A$ are cohomologically stable and correspond to the specialisation
  closed subsets of $\Spec\La$.  This goes back to work of Gabriel in
  \cite{Ga1962}; it was extended by Takahashi to a correspondence for
  all cohomologically stable subcategories in \cite{Ta2009} and more
  recently in \cite{MT2022}, using Neeman's classification of
  localising subcategories of the derived category $\bfD(\A)$ from
  \cite{Ne1992}. A generalisation of Gabriel's correspondence for
  locally noetherian Grothendieck categories was established in
  \cite{WM2024}.

(2) For a locally coherent Grothendieck category $\A$, the localising
subcategories of finite type are classified in terms of $\Sp\A$; see
\cite{He1997,Kr1997}. This is somewhat orthogonal to Gabriel's
correspondence.  More precisely, a localising subcategory
$\C\subseteq\A$ is of \emph{finite type} when the corresponding
subcategory of torsion-free objects
\[\C^\perp=\{X\in\A\mid \Hom(C,X)=0\text{ for all
  }C\in\C\}\] is closed under filtered colimits. The subcategory
$\C^\perp$ is an essentially closed subcategory, and both $\C^\perp\cap\Sp\A$ and
$P(\C^\perp)\in\bfS(\A)$ determine $\C$, while $\C$ need not be
cohomologically stable; cf.\ Theorem~\ref{th:spectral-localising}.
\end{rem}

\subsection*{Support data}

The following is the analogue of Definition~\ref{de:support}. The only
difference appears in the last condition (E$\wedge$) where the
intersection is based on cohomologically stable subcategories, whereas
in condition (S$\wedge$) essentially closed subcategories are used. 

\begin{defn}
An \emph{exact support datum} on a Grothendieck category $\A$ is a pair
$(T,\s)$ consisting of a topological space $T$ and a map
$\s\colon\Ob\A\to \Clop(T)$ such that
\begin{enumerate}
\item[(E$0$)] $\s(X)=\varnothing$ for $X=0$,
\item[(E$1$)] $\s(X)=T$ for every generator $X$ of $\A$,
\item[(E$\vee$)] $\s(X)\cup\s(Y)=\s(X\oplus Y)$ for all objects $X,Y$ in
  $\A$, and
\item[(E$\wedge$)] $\s(X)\cap\s(Y)=\s(Z)$ for all objects $X,Y,Z$ in $\A$
  with $\langle X\rangle_\ex\cap \langle Y\rangle_\ex=\langle Z\rangle_\ex$.
\end{enumerate}
A morphism $(T,\s)\to (T',\s')$ is given by a continuous map $f\colon
T\to T'$ such that $\s(X)=f^{-1}(\s'(X))$ for all $X\in\A$.
\end{defn}

The following is the exact analogue of Theorem~\ref{th:support}; it has the same proof.

\begin{thm}
  \pushQED{\qed}
  The pair $(\Spc\A,\suppex)$ is an exact support datum on $\A$. Moreover,
  for any exact support datum $(T,\s)$ there is a unique continuous map
  $f\colon T\to\Spc\A$ such that $\s(X)=f^{-1}(\suppex(X))$ for all
  $X\in\A$. Explicitly, the map $f$ is given by the adjunction
  \eqref{eq:stone} and therefore
  \[f(p)=\{\langle X\rangle\in\bfS(\A)\mid p\not\in\s(X)\}\qquad\text{for
    }p\in T.\qedhere\] 
\end{thm}

The computation of $\suppex$ in  Lemma~\ref{le:res} shows that for each object $X\in\A$ we
have an inclusion
\[\supp(X)\subseteq\suppex(X).\]
There is an equality when $\A$ is spectral, but small examples show
that it is a proper inclusion in general.

\begin{exm}
Let $\A=\Mod\La$ be the module category of $\La=\smatrix{k&k\\0&k}$
where $k$ is a field. Then there are two indecomposable injective
$\La$-modules and $\Sp\A=\{I_1,I_2\}=\Spc\A$. There is a short exact
sequence $0\to S_1\to I_1\to I_2\to 0$ where $S_1=\soc(I_1)$, and therefore
\[\supp(S_1)=\{I_1\}\qquad\text{and}\qquad \suppex(S_1)=\{I_1,I_2\}.\]
\end{exm}

\subsection*{Stable localising subcategories}

Let $\A$ be a Grothendieck category. A localising subcategory is called
\emph{stable} if it is closed under essential extensions, equivalently
if it is closed under injective envelopes. The stable localising
subcategories are partially ordered by inclusion and closed under all
intersections; so they form a complete lattice.

Given an element $\U\in\bfS(\A)$, we have always the inclusion
$\A_\U\subseteq P^{-1}(\U)$. So one may ask when the equality
$\A_\U=P^{-1}(\U)$ holds. 

\begin{lem}\label{le:coh}
  For  a full subcategory $\C\subseteq\A$ the following are
  equivalent.
  \begin{enumerate}
\item $\C$  is stable localising. 
\item $\C$ is essentially closed and cohomologically stable.
\item There exists $\U\in\bfS(\A)$ such that $\C=\A_\U=P^{-1}(\U)$.
  \end{enumerate}
\end{lem}
\begin{proof}
 If $\C\subseteq\A$ is stable localising, then $\C=P^{-1}(\U)$ for
 some $\U\in\bfS(\A)$. Also, we have $\C=\A_\U$ since $\C$ is
 cohomologically stable. On the other hand, when $\A_\U=P^{-1}(\U)$,
 then this subcategory is clearly stable localising.
\end{proof}

\begin{prop}\label{pr:stable-loc}
  Let $\A$ be a Grothendieck category. The assignment
  \[\C\longmapsto\supp(\C)=P(\C)\] induces an embedding 
\[\{\C\subseteq\A\mid \C\text{  stable
    localising}\}\hookrightarrow\bfS(\A)\]
which preserves  arbitrary joins and meets. 
\end{prop}
\begin{proof}
  We apply Lemma~\ref{le:coh}.  Each stable localising subcategory is
  essentially closed. Thus the embedding follows from Theorem~\ref{th:spectral}.
  The stable localising subcategories of $\A$ are closed under
  arbitrary intersections. So arbitrary meets are preserved. For a
  family $(\C_\a)_{\a\in A}$ of stable localising subcategories with
  $\U_\a=P(\C_\a)$ we consider the join $\U=\bigvee_\a \U_\a$ and need to show
  that $P^{-1}(\U)$ is a stable localising subcategory. The
  subcategory is essentially closed. So it suffices to show that for any
  epimorphism $\p\colon I\to X$ with $I$ injective and $P(I)\in \U$ we
  have that $P(X)\in \U$. We begin with the finite case and may assume
  $A= \{\a,\b\}$. Let $I=I_\a\oplus I_\b$ with $I_\a\in\C_\a$ and
  $I_\b\in\C_\b$. This decomposition induces an exact sequence
  $0\to X_\a\to X\to X_\b\to 0$ where $X_\a=\p(I_\a)$ and $X_\b$ is a
  quotient of $I_\b$. Since $\C_\a$ and $\C_\b$ are closed under
  quotients, it follows that $P(X)\in \U$. Writing $A$ as directed
  union of its finite subsets, we may assume that
  $\U=\bigvee_\a \U_\a$ is directed. Then $I=\sum_\a I_\a$ is the
  directed union of injectives $I_\a\in\C_\a$ by
  Lemma~\ref{le:join}. This implies $X=\sum_\a X_\a$, where
  $X_\a=\p(I_\a)$. We have $P(X)=\sum_\a P(X_\a)$, and this belongs to
  $\U$ since $P(X_\a)\in\U_\a$ for all $\a$.
\end{proof}

\begin{cor}
The stable localising subcategories of a Grothendieck category form a
frame and a coframe.
\end{cor}
\begin{proof}
  This follows from the above theorem since $\bfS(\A)$ is a frame and
  a coframe, by Propositions~\ref{pr:lattice-loc} and
  \ref{pr:lattice-loc-spec}.
\end{proof}

\begin{exm}
Let $\A=\Mod\La$ be the module category of a ring $\La$. When $\La$ is
commutative noetherian, then all localising subcategories are
stable. This fails when $\La$ is not commutative. For
$\La=\smatrix{k&k\\0&k}$ and $k$ a field, there are two proper
localising subcategories and one of them is not stable.
\end{exm}

\subsection*{Derived categories}

For a Grothendieck category $\A$ we have already considered the right
derived functor of the canonical functor $P\colon\A\to\Spec\A$. The
restriction $\Inj\A\to(\Inj\A)/\Rad(\Inj\A)\iso\Spec\A$ induces an
exact functor
\[\bfK(\Inj\A)\lto\bfK(\Spec\A)\iso\bfD(\Spec\A)\]
where $\bfK(\Inj\A)$ denotes the category of complexes in $\Inj\A$ up
to homotopy.  Its computation for any $X\in\bfK(\Inj\A)$ is very
explicit. By removing non-zero contractible summands we may assume
that the complex $X$ is homotopically minimal so that all
differentials are radical morphisms; see
\cite[Proposition~4.3.18]{Kr2022}. Then $P(X)$ has zero differentials.

For example, when $\A$ is locally noetherian the triangulated category
$\bfK(\Inj\A)$ is compactly generated with
category of compacts equivalent to $\bfD^{\mathrm b}(\Noeth\A)$, where
$\Noeth\A$ denotes the full subcategory of noetherian objects in $\A$;
see \cite{Kr2005}. Moreover, in this case the functor
$\bfK(\Inj\A)\to\bfD(\Spec\A)$ preserves all coproducts.

Given a triangulated category $\T$ with coproducts, let $\bfL(\T)$
denote the lattice of localising subcategories. The inclusion
$i\colon\Spec\A\to\bfD(\Spec\A)$ induces a lattice isomorphism
\[\bfL(\bfD(\Spec\A))\longiso\bfL(\Spec\A)\]
by taking $\U\subseteq \bfD(\Spec\A)$ to $i^{-1}(\U)$. The inverse map
sends $\V\subseteq\Spec\A$ to $\bfD(\V)$ where we identify
\[\bfD(\V)=\{X\in\bfD(\Spec\A)\mid  H^n(X)\in\V\text{ for all
  }n\in\bbZ\}.\] This yields a notion of support for objects
$X\in\bfK(\Inj\A)$ with $\supp(X)\subseteq\Spc\A$, since
$\bfL(\Spec\A)\iso\Clop(\Spc\A)$. Details and properties of this are
left to the interested reader.
  
\section{Exactly definable categories}\label{se:ext}

There are interesting classes of rings where the Boolean spectrum of
its module category is a finite set with the discrete topology. For
instance this includes all artinian rings. In such a case it is 
more interesting to study the pure-injective modules and the
corresponding pure spectral category, as the class of injective
modules is rather small.  The notion of purity for modules goes back to
Cohn \cite{Co1959} and was later generalised by Crawley-Boevey in the
context of locally finitely presented additive categories
\cite{CB1994}. A key idea is to identify the relevant additive
category $\C$ with the full subcategory of fp-injective objects in a
much larger locally coherent Grothendieck category $\bfP(\C)$, the
purity category of $\C$. Following \cite{Kr1998} such categories $\C$
are called \emph{exactly definable}, because they identify with the
category of exact functors $\A\to\Ab$ from a small abelian category
$\A$ to the category of abelian groups.

\subsection*{Exactly definable categories}\label{se:ex-definable}

Let $\A$ be a Grothendieck category that is \emph{locally
  coherent}. Thus $\A$ has a generating set of finitely presented
objects and the full subcategory $\fp\A$ of finitely presented objects
is abelian. An object $X\in\A$ is \emph{fp-injective} if
$\Ext^1(F,X)=0$ for all $F\in\fp\A$.  Let $\C:=\Fpinj\A$ denote the
full subcategory of fp-injective objects in $\A$. This is an exact
subcategory of $\A$ that is closed under filtered colimits and
products. In fact, the exact structure of $\C$ is intrinsic, because
it is the smallest exact structure on $\C$ such that the exact
sequences are closed under filtered colimits.\footnote{I am grateful
  to Kevin Schlegel for suggesting the argument: For any object
  $X\in\C$ there is a reduced power $\bar X=\prod_I X/\F$ (given by an
  appropriate ultrafilter $\F$ on some sufficiently large set $I$)
  which is injective in $\A$; see \cite[4.2.5]{Pr2009}. The canonical
  morphism $X\to \bar X$ is by construction a filtered colimit of
  split monos, and it factors through any mono $X\to Y$ since $\bar X$
  is injective. Thus any mono $X\to Y$ is an admissible mono with
  respect to any exact structure that is closed under filtered
  colimits.} The exact sequences in $\C$ are called \emph{pure-exact},
and $\C$ admits pure-injective envelopes which identify with the
injective envelopes in the ambient category $\A$.

\begin{defn}\label{de:ex-def}
  A category that is equivalent to one of the form $\Fpinj\A$ for a
  locally coherent category $\A$ is called \emph{exactly definable};
  it has filtered colimits, products and an intrinsic exact structure
  given by the pure-exact sequences.
\end{defn}

Let $\C$ be an exactly definable category, given by a locally coherent
category $\A$. Then $\A$ is uniquely determined by $\C$, because the
category of pure-injective objects in $\C$ identifies with the
category of injective objects in $\A$, and any abelian category with
enough injective objects is uniquely determined by its category of
injective objects. We write $\bfP(\C)$ for $\A$ and call it the
\emph{purity category} of $\C$.

\begin{exm}\label{ex:ex-def}
Let $\La$ be a ring and let $\mmod\La$ denote its category of finitely presented
modules. The category $(\mmod(\La^\op),\Ab)$ of additive
functors $\mmod(\La^\op)\to\Ab$ is a locally coherent Grothendieck
category, and the assignment \[X\longmapsto X\otimes_\La
  -|_{\mmod(\La^\op)}\]
identifies $\Mod\La$ with the full subcategory of fp-injective objects
in $(\mmod(\La^\op),\Ab)$. Thus $\Mod\La$ is an exactly definable
category and
\[\bfP(\Mod\La)=(\mmod(\La^\op),\Ab).\]
This goes back to Gruson and Jensen
\cite{GJ1981}; see also  \cite[\S3]{CB1994} or \cite[12.4]{Kr2022}.
\end{exm}

\subsection*{Definable subcategories}

Fix an exactly definable category $\C$, given by a locally coherent
category $\A$.  Following Crawley-Boevey \cite{CB1998} a full
subcategory $\D\subseteq\C$ is called \emph{definable} if it is closed
under filtered colimits, products, and pure subobjects. The
terminology is justified by the fact that definable subcategories of
$\C$ correspond to Serre subcategories of $\fp\A$: we have
\[\D=\{X\in\C\mid\Hom(F,X)=0\text{ for all }F\in\calS\}\] for some Serre
subcategory $\calS\subseteq\fp\A$; see \cite[12.2]{Kr2022} for
details.  From the above description it follows that $\D$ is a thick
subcategory. Moreover, $\D$ is closed under pure-injective envelopes;
see \cite[Proposition~12.2.8]{Kr2022}. We generalise the definition of
a definable subcategory in two directions, following the
Definitions~\ref{de:spectral-subcat} and \ref{de:stable-subcat}.

\begin{defn}
  A full subcategory $\D\subseteq\C$ is called \emph{pure-essentially closed} if
  it is closed under arbitrary coproducts, pure subobjects, and  pure-injective
  envelopes.
\end{defn}

\begin{defn}
  A thick subcategory $\D\subseteq\C$ is called \emph{pure-cohomologically stable} if
  \begin{enumerate}
\item for every family $(X_\a)$ of objects in $\D$ the pure-injective
  envelope of  $\coprod_\a X_\a$ belongs to $\D$, and
\item for every pure-exact sequence $X^0\to X^1\to X^2\to\cdots$ with all
  $X^n$ in $\D$ the kernel of $X^0\to X^1$ belongs to $\D$.
\end{enumerate}
\end{defn}

For an exact category we recall that a sequence of composable morphisms
$(\p^n)$  is \emph{exact} if each $\p^n$ is the composite
$\p^n=\iota_n\pi_n$ of an admissible epimorphism (deflation) followed
by an admissible monomorphism (inflation) so that each pair
$(\iota_n,\pi_{n+1})$ is a conflation.

\begin{lem}
  Any definable subcategory of an exactly definable category is
  pure-essentially closed and pure-cohomologically stable.
\end{lem}
\begin{proof}
  This follows from the defining properties of a definable subcategory
  and the additional properties that have already been
  mentioned. Specifically, one uses that a definable subcategory is
  thick and closed under pure-injective envelopes.
\end{proof}

The pure-essentially closed and the pure-cohomologically stable subcategories are partially
ordered by inclusion and closed under arbitrary intersections; so they
form complete lattices.

\subsection*{The Ziegler topology}

The class of pure-injective modules has been studied in seminal work
of Ziegler on the model theory of modules \cite{Zi1984}. In
particular, he introduced for any ring a topology on the set of isoclasses of
indecomposable pure-injective modules.

Let $\Ind\C$ denote the set of isoclasses of indecomposable
pure-injective objects in $\C$; this identifies with $\Sp\bfP(\C)$. The assignment
\begin{equation}\label{eq:Ziegler}
  \D\longmapsto\D\cap\Ind\C
\end{equation}  identifies the 
definable subcategories of $\C$ with the \emph{Ziegler closed} subsets
of $\Ind\C$. The inverse map sends a Ziegler closed subset $\U$ to the
definable subcategory  generated by $\U$, which consists of
all pure subobjects of products of objects in $\U$. This
correspondence is due to Crawley-Boevey for module categories
\cite[2.5]{CB1998} and we refer to \cite[12.2]{Kr2022} for the general version.

\subsection*{The Boolean spectrum}

Let $\PInj\C$ denote the full subcategory of pure-injective objects in
$\C$; it identifies with $\Inj\bfP(\C)$ via the embedding
$\C\to\bfP(\C)$. The \emph{pure spectral category} is the quotient
\[\PSpec\C:=(\PInj\C)/\Rad(\PInj\C)\iso\Spec\bfP(\C)\]
which is a spectral Grothendieck category because of the equivalence
from \eqref{eq:Spec}. We consider the lattice of
localising subcategories
\[\bfL(\PSpec\C)\iso\bfS(\bfP(\C)).\]
This is a Boolean lattice by Proposition~\ref{pr:Boole} and  its spectrum
\[\PSpc\C:=\Spec (\bfL(\PSpec\C))\]
can be written as the disjoint union of a discrete and a continuous
part as in \eqref{eq:Spec-decomposition}. This yields a  canonical embedding
$\Ind\C\hookrightarrow\PSpc\C$ 
which identifies $\Ind\C$ with the discrete part. We endow $\Ind\C$
with the Ziegler topology and write $\Cl(\Ind\C)$ for its lattice of
closed subsets.

In the following we provide two possible extensions of
Crawley-Boevey's correspondence \eqref{eq:Ziegler} for definable
subcategories, which amounts to a lattice isomorphism
  \[ \{\D\subseteq\C\mid\D\text{ definable}\}\longiso\Cl(\Ind\C).\]
We view $\C$ as a full subcategory of the purity
category $\bfP(\C)$ so that $\supp(X)$ and $\suppex(X)$ for objects
$X\in\C$ are defined via \eqref{eq:supp} and
\eqref{eq:suppex}. Moreover, these are viewed as subsets of $\PSpc\C$
via the isomorphism
\begin{equation}\label{eq:PSpc}
  \bfS(\bfP(\C))\iso\bfL(\PSpec\C)\xto{\eqref{eq:stone-unit}} \Clop(\PSpc\C).
\end{equation}

\begin{thm}\label{th:pure-spectral}
  Let $\C$ be an exactly definable category.  The assignment $\D\mapsto\supp(\D)$
  induces a lattice isomorphism
  \[ \{\D\subseteq\C\mid\D\text{ pure-essentially closed}\}\longiso\Clop(\PSpc\C).\] Moreover, the embedding
  $\Ind\C\hookrightarrow\PSpc\C$ identifies $\D\cap\Ind\C$ with
  $\supp(\D)\cap\Ind\C$ when $\D\subseteq\C$ is definable.
\end{thm}
\begin{proof}
  We view $\C\subseteq\bfP(\C)$ as a full subcategory. Taking an
  essentially closed subcategory $\U\subseteq\bfP(\C)$ to $\U\cap\C$ induces a
  bijection between the essentially closed subcategories of $\bfP(\C)$ and the
  pure-essentially closed subcategories of $\C$. This follows from the fact that
  any essentially closed subcategory $\U\subseteq\bfP(\C)$ is determined by
  $\U\cap \Inj\bfP(\C)$, keeping in mind that
  $\PInj\C=\Inj\bfP(\C)$. Given this correspondence the isomorphism is
  obtained by composing the isomorphisms from
  Theorem~\ref{th:spectral} and \eqref{eq:PSpc}. The additional
  assertion about definable subcategories follows by unraveling the
  construction of the embedding $\Ind\C\hookrightarrow\PSpc\C$.
\end{proof}  

\begin{thm}\label{th:pure-stable}
  Let $\C$ be an exactly definable category.  The assignment
  $\D\mapsto\suppex(\D)$  induces a
  lattice isomorphism
  \[ \{\D\subseteq\C\mid\D\text{ pure-cohomologically stable}\}\longiso\Clop(\PSpc\C).\] Moreover, the embedding
  $\Ind\C\hookrightarrow\PSpc\C$ identifies $\D\cap\Ind\C$ with
  $\suppex(\D)\cap\Ind\C$ when $\D\subseteq\C$ is definable. 
\end{thm}
\begin{proof}
Adapt the proof of Theorem~\ref{th:pure-spectral} by using
Theorem~\ref{th:cohstable} instead of Theorem~\ref{th:spectral}.
\end{proof}

\begin{rem}
  For every definable
  subcategory $\D\subseteq\C$ we have $\supp(\D)=\suppex(\D)$.
\end{rem}

\subsection*{Module categories}

Let $\La$ be a ring and consider its module category $\C=\Mod\La$
which is exactly definable. We use the obvious notation and set
$\PInj\La=\PInj\C$, $\Ind\La=\Ind\C$, and $\PSpc\La=\PSpc\C$. The
books of Prest \cite{Pr1988} and also Jensen and Lenzing \cite{JL1989}
discuss the basic structure of the \emph{pure spectral
  category} \[\PSpec\La=(\PInj\La)/\Rad(\PInj\La).\] For instance, it
is shown that any pure-injective module $X$ admits an essentially
unique decomposition $X=X_{\mathrm d}\oplus X_{\mathrm c}$ into a
\emph{discrete} and a \emph{continuous} part
\cite[Corollary~8.28]{JL1989}. 
This reflects the decomposition
\begin{equation*}\label{eq:pure-spec}
  \PSpec\La=(\PSpec\La)_{\mathrm
    d}\times (\PSpec\La)_{\mathrm c}
\end{equation*}
of the spectral category.  For a more detailed treatment of
decompositions of pure-injective modules we refer to
\cite{Fa1985}. Unfortunately, not much seems to be known about the
continuous part of the spectral category, except that one knows for
various classes of rings when the continuous part vanishes. On the
other hand, superdecomposable modules have been constructed more or
less explicitly for several rings, but these constructions do not
provide much insight into the ways they decompose; see \cite{Ri2000}
for a survey. In any case, the discussion in \S\ref{se:decomp}
explains how the Boolean lattice $\bfL(\PSpec\La)$ controls the
decompositions of pure-injective $\La$-modules; see also
Example~\ref{ex:superdecomp}.

We characterise the absence of superdecomposable modules and may ask
which Boolean lattices arise from representations of finite
dimensional algebras, following ideas from \cite{Br1974}.

\begin{prop}
We have  $X=X_{\mathrm d}$ for every pure-injective
$\La$-module $X$ if and only if the canonical embedding
$\Ind\La\hookrightarrow\PSpc\La$ is a bijection.
\end{prop}
\begin{proof}
This follows from the above decomposition of the pure spectral
category $\PSpec\La$
and the corresponding decomposition of the Boolean spectrum
\eqref{eq:Spec-decomposition}.
\end{proof}

\begin{probl}
  Let $\La$ be a finite dimensional algebra of wild representation
  type (eg.\ the path algebra of the $3$-Kronecker quiver) and $L$ a
  complete Boolean lattice. Is there a pure-injective $\La$-module $X$
  such that its lattice of direct summands $\bfD(X)$ is isomorphic to
  $L$?
\end{probl}

The following result complements our extension of Crawley-Boevey's
correspondence for definable subcategories.

\begin{prop}\label{pr:fin-type}
  For an Artin algebra $\La$ the following are equivalent.
  \begin{enumerate}
  \item  The embedding $\Cl(\Ind\La)\to \Clop(\PSpc\La)$ is a bijection.
  \item Every pure-essentially closed subcategory of $\Mod\La$ is definable. 
  \item Every pure-cohomologically stable subcategory of $\Mod\La$ is definable. 
  \item The Ziegler spectrum $\Ind\La$ is a discrete space.
  \item The algebra $\La$ is of finite representation type.
  \end{enumerate}
\end{prop}
\begin{proof}
  (1) $\Leftrightarrow$ (2): This follows from Theorem~\ref{th:pure-spectral}. 

  (1) $\Leftrightarrow$ (3): This follows from Theorem~\ref{th:pure-stable}. 

  (2) $\Rightarrow$ (4): Every subset of $\PSpc\La$ contained in
  $\Ind\La$ is clopen, by Lemma~\ref{le:discrete}. It follows from
  Theorem~\ref{th:pure-spectral} that every subset of $\Ind\La$ is
  Ziegler closed.

  (4) $\Rightarrow$ (5): The space $\Ind\La$ is finite when it is
  discrete, because $\Ind\La$ is quasi-compact. For an Artin algebra,
  every indecomposable finitely presented module is pure-injective, so
  $\La$ is of finite representation type.

  (5) $\Rightarrow$ (2): When $\La$ is of finite representation type,
  then every $\La$-module is endofinite, and in particular
  pure-injective \cite[Theorem~13.2.10]{Kr2022}.  Let $\D$ be a
  pure-essentially closed subcategory. Then it follows that $\D$ is
  closed under pure subobjects and pure quotients. In particular, $\D$
  is closed under filtered colimits because any filtered colimit of
  objects in $\D$ is a pure quotient of a coproduct of objects in
  $\D$. Finite representation type implies that every $\La$-module
  decomposes into a coproduct of indecomposables, and every product of
  objects in $\D$ decomposes into indecomposables from $\D$, because
  only finitely many indecomposables are involved. This follows again
  from endofiniteness. Thus $\D$ is closed under products.
\end{proof}

We close our discussion of module categories with some remarks.

\begin{rem}
  Beyond finite representation type the correspondences for
  pure-essentially closed and pure-cohomologically stable subcategories
  (Theorems~\ref{th:pure-spectral} and \ref{th:pure-stable}) seem to
  be new. In particular, Proposition~\ref{pr:fin-type} shows that the
  extension of the Ziegler spectrum yields additional information
  whenever the ring is not of finite representation type. This is
  illustrated by looking at the category of abelian groups
  $\Ab=\Mod\bbZ$. Consider the Pr\"ufer group $\bbZ(p^\infty)$ given
  by a prime $p$. The direct sums of copies of $\bbZ(p^\infty)$ form a
  subcategory of $\Ab$ which is pure-essentially closed and pure-cohomologically stable, but not
  definable, because it is not closed under products. More precisely,
  the rationals $\bbQ$ arise as direct summand of a product of copies
  of $\bbZ(p^\infty)$, and then $\{\bbZ(p^\infty),\bbQ\}$ is the
  Ziegler closure of $\bbZ(p^\infty)$.

  The Ziegler spectrum plus all definable subcategories are well
  understood for various classes of rings, including the rings of
  finite representation type, commutative Dedekind domains, or tame
  hereditary algebras; see \cite{Pr2009} for a survey. For these rings
  it seems reasonable to expect similar descriptions for the
  pure-essentially closed and the pure-cohomologically stable
  subcategories. Note that in these examples $\Ind\La=\PSpc\La$; so
  the relevant subcategories are parameterised by all subsets of
  $\Ind\La$.
\end{rem}
  
\subsection*{Acknowledgements} 
It is a pleasure to thank Greg Stevenson for several discussions and
helpful comments on this work. Also, I am grateful to Ken Goodearl and
Jan Šťovíček for their valuable input. Further thanks are due to
an anonymous referee. The work was supported by the Deutsche
Forschungsgemeinschaft (Project-ID 491392403 – TRR~358).


\begin{thebibliography}{10}

\bibitem{AS1974} E.~P. Armendariz\ and\ S.~A. Steinberg, Regular
  self-injective rings with a polynomial identity,
  Trans. Amer. Math. Soc. {\bf 190} (1974), 417--425.

\bibitem{Ba2005} P. Balmer, \emph{The spectrum of prime ideals in
    tensor triangulated categories}, J. Reine
  Angew. Math. \textbf{588} (2005), 149--168.

\bibitem{BK1987} F. Borceux\ and\ G.~M. Kelly, On locales of
  localizations, J. Pure Appl. Algebra {\bf 46} (1987), no.~1, 1--34.

\bibitem{BL1959} B. Brainerd\ and\ J. Lambek, On the ring of quotients
  of a Boolean ring, Canad. Math. Bull. {\bf 2} (1959), 25--29.

\bibitem{Br2018} M. Brandenburg, Rosenberg's reconstruction theorem,
  Expo. Math. {\bf 36} (2018), no.~1, 98--117.

\bibitem{Br1974} S. Brenner, Decomposition properties of some small
  diagrams of modules, in {\it Symposia Mathematica, Vol. XIII
    (Convegno di Gruppi Abeliani \& Convegno di Gruppi e loro
    Rappresentazioni, INDAM, Rome, 1972)}, 127--141, Academic Press,
  London, 1974.

\bibitem{Co1959} P.~M. Cohn, On the free product of associative rings,
  Math. Z. {\bf 71} (1959), 380--398.

\bibitem{CB1994} W. Crawley-Boevey, Locally finitely presented
  additive categories, Comm. Algebra {\bf 22} (1994), no.~5,
  1641--1674.

\bibitem{CB1998} W. Crawley-Boevey, Infinite-dimensional modules in
  the representation theory of finite-dimensional algebras, in {\it
    Algebras and modules, I (Trondheim, 1996)}, 29--54, CMS
  Conf. Proc., 23, Amer. Math. Soc., Providence, RI, 1998.

\bibitem{Da1997} J. Dauns, Module types, Rocky Mountain J. Math. {\bf
    27} (1997), no.~2, 503--557.
  
\bibitem{DST2019} M.~A. Dickmann, N. Schwartz\ and\ M. Tressl, {\it
    Spectral spaces}, New Mathematical Monographs, 35, Cambridge
  Univ. Press, Cambridge, 2019.

\bibitem{Fa1983} A. Facchini, Spectral categories and varieties of
  preadditive categories, J. Pure Appl. Algebra {\bf 29} (1983),
  no.~3, 219--239.

\bibitem{Fa1985} A. Facchini, Decompositions of algebraically compact
  modules, Pacific J. Math. {\bf 116} (1985), no.~1, 25--37.

\bibitem{Ga1962} P. Gabriel, Des cat\'egories ab\'eliennes,
  Bull. Soc. Math. France {\bf 90} (1962), 323--448.

\bibitem{GO1966} P. Gabriel\ and\ U. Oberst, Spektralkategorien und
  regul\"{a}re Ringe im von-Neumannschen Sinn, Math. Z. {\bf 92}
  (1966), 389--395.

\bibitem{Go1986} J.~S. Golan, {\it Torsion theories}, Pitman
  Monographs and Surveys in Pure and Applied Mathematics, 29, Longman
  Sci. Tech., Harlow, 1986.

\bibitem{Go1979} K.~R. Goodearl, {\it von Neumann regular rings},
  Monographs and Studies in Mathematics, 4, Pitman, Boston, MA, 1979.

\bibitem{GB1976} K.~R. Goodearl\ and\ A.~K. Boyle, Dimension theory
  for nonsingular injective modules, Mem. Amer. Math. Soc. {\bf 7}
  (1976), no.~177, {\rm viii}+112 pp.

\bibitem{GW2005} K.~R. Goodearl\ and\ F. Wehrung, The complete
  dimension theory of partially ordered systems with equivalence and
  orthogonality, Mem. Amer. Math. Soc. {\bf 176} (2005), no.~831,
  vii+117 pp.
  
\bibitem{GJ1981} L. Gruson\ and\ C.~U. Jensen, Dimensions
  cohomologiques reli\'{e}es aux foncteurs $\varprojlim\sp{(i)}$, in
  {\it Paul Dubreil and Marie-Paule Malliavin Algebra Seminar, 33rd
    Year (Paris, 1980)}, 234--294, Lecture Notes in Math., 867,
  Springer, Berlin, 1981.

\bibitem{He1997} I. Herzog, The Ziegler spectrum of a locally coherent
  Grothendieck category, Proc. London Math. Soc. (3) {\bf 74} (1997),
  no.~3, 503--558.

\bibitem{JL1989} C.~U. Jensen\ and\ H. Lenzing, {\it Model-theoretic
    algebra with particular emphasis on fields, rings, modules},
  Algebra, Logic and Applications, 2, Gordon and Breach, New York,
  1989.

\bibitem{Jo1982} P. T. Johnstone, \emph{Stone spaces}, Cambridge
  Stud. Adv. Math., vol.~3, Cambridge Univ. Press, Cambridge, 1982.

\bibitem{Ka2015} R. Kanda, Classification of categorical subspaces of
  locally Noetherian schemes, Doc. Math. {\bf 20} (2015), 1403--1465.

\bibitem{Ka1951} I. Kaplansky, Projections in Banach algebras, Ann. of
  Math. (2) {\bf 53} (1951), 235--249.
  
\bibitem{Kr1997} H. Krause, The spectrum of a locally coherent
  category, J. Pure Appl. Algebra {\bf 114} (1997), no.~3, 259--271.

\bibitem{Kr1998} H. Krause, Exactly definable categories, J. Algebra
  {\bf 201} (1998), no.~2, 456--492.

\bibitem{Kr2005} H. Krause, The stable derived category of a
  Noetherian scheme, Compos. Math. {\bf 141} (2005), no.~5,
  1128--1162.  

\bibitem{Kr2022} H. Krause, {\it Homological theory of representations}, 
 Cambridge Studies in Advanced Mathematics, Cambridge University
 Press,  195, Cambridge, 2022.

\bibitem{MT2022} H. Matsui\ and\ R. Takahashi, Filtrations in module
  categories, derived categories, and prime spectra,
  Int. Math. Res. Not. IMRN {\bf 2022}, no.~5, 3457--3492.

\bibitem{MvN1936} F.~J. Murray\ and\ J. von~Neumann, On rings of
  operators, Ann. of Math. (2) {\bf 37} (1936), no.~1, 116--229.

\bibitem{Ne1992} A. Neeman, The chromatic tower for $D(R)$, Topology {\bf
  31} (1992), no.~3, 519--532.

\bibitem{Pr1988} M. Prest, {\it Model theory and modules}, London
  Mathematical Society Lecture Note Series, 130, Cambridge
  Univ. Press, Cambridge, 1988.

\bibitem{Pr2009} M. Prest, {\it Purity, spectra and localisation},
  Encyclopedia of Mathematics and its Applications, 121, Cambridge
  Univ. Press, Cambridge, 2009.

\bibitem{Ro1967} J.-E. Roos, Locally distributive spectral categories
  and strongly regular rings, in {\it Reports of the Midwest Category
    Seminar}, 156--181, Lecture Notes in Math., No. 47, Springer,
  Berlin, 1967.

\bibitem{Ri2000} C.~M. Ringel, Infinite length modules. Some examples
  as introduction, in {\it Infinite length modules (Bielefeld, 1998)},
  1--73, Trends Math, Birkh\"{a}user, Basel, 2000.
  
\bibitem{Ro1995} A.~L. Rosenberg, {\it Noncommutative algebraic
    geometry and representations of quantized algebras}, Mathematics
  and its Applications, 330, Kluwer Acad. Publ., Dordrecht, 1995.

\bibitem{St1975} B.~T. Stenstr\"{o}m, {\it Rings of quotients}, Die
  Grundlehren der mathematischen Wissenschaften, Band 217, Springer,
  New York, 1975.

\bibitem{St2024} G. Stevenson, Localizing subcategories of
  Grothendieck categories, Preprint, 2024.

\bibitem{Ta2009} R. Takahashi, On localizing subcategories of derived
  categories, J. Math. Kyoto Univ. {\bf 49} (2009), no.~4, 771--783.
  
\bibitem{WM2024} K. Wu\ and\ X. Ma, A classification of some thick
  subcategories in locally noetherian Grothendieck categories,
  Glasg. Math. J. {\bf 66} (2024), no.~1, 175--182.

\bibitem{Zi1984} M. Ziegler, Model theory of modules, Ann. Pure
  Appl. Logic {\bf 26} (1984), no.~2, 149--213.

\end{thebibliography}
\end{document}